\documentclass[11pt,english]{article}
\usepackage[T1]{fontenc}
\usepackage[latin9]{inputenc}
\usepackage{geometry}
\usepackage{babel}
\usepackage{mathtools}
\usepackage{url}
\usepackage{amsmath}
\usepackage{amsthm}
\usepackage{amssymb}
\usepackage[unicode=true,pdfusetitle,
 bookmarks=true,bookmarksnumbered=false,bookmarksopen=false,
 breaklinks=false,pdfborder={0 0 1},backref=false,colorlinks=false]
 {hyperref}

\makeatletter
\numberwithin{equation}{section}
\numberwithin{figure}{section}
\theoremstyle{plain}
\newtheorem{thm}{\protect\theoremname}
\theoremstyle{plain}

\theoremstyle{plain}
\newtheorem{lem}[thm]{\protect\lemmaname}
\theoremstyle{remark}
\newtheorem{rem}[thm]{\protect\remarkname}
\theoremstyle{plain}
\newtheorem{prop}[thm]{\protect\propositionname}
\theoremstyle{remark}

\theoremstyle{definition}

\theoremstyle{plain}
\newtheorem{cor}[thm]{Corollary}

\date{}
\usepackage{tikz-cd}
\usepackage{wasysym}
\usepackage{MnSymbol}
\usepackage[tt=false]{libertine}
\usepackage[parfill]{parskip}
\usepackage{enumitem}
\setlist[itemize]{noitemsep,topsep=5pt}

\usepackage{titlesec}
\titleformat{\section}{\large\bfseries\filleft}{\thesection}{1em}{}[{\titlerule[0.8pt]}]

\renewcommand\labelenumi{(\roman{enumi})}
\renewcommand\theenumi\labelenumi

\DeclareMathOperator{\Spec}{Spec}

\DeclareMathOperator{\Pic}{Pic}

\DeclareMathOperator{\Mor}{Mor}

\DeclareMathOperator{\PGL}{PGL}

\DeclareMathOperator{\Gr}{Gr}

\let\oldtheorem\thm
\renewcommand{\thm}{\oldtheorem\normalfont}
\let\oldprop\prop
\renewcommand{\prop}{\oldprop\normalfont}
\let\oldcor\cor
\renewcommand{\cor}{\oldcor\normalfont}
\let\oldlem\lem
\renewcommand{\lem}{\oldlem\normalfont}
\newenvironment{ack}{\textit{Acknowledgements.}}{}
\makeatother

\providecommand{\claimname}{Claim}
\providecommand{\definitionname}{Definition}
\providecommand{\lemmaname}{Lemma}
\providecommand{\propositionname}{Proposition}
\providecommand{\questionname}{Question}
\providecommand{\remarkname}{Remark}
\providecommand{\theoremname}{Theorem}

\usepackage[style=numeric, sorting=nyt, doi=false,isbn=false, maxbibnames=99, giveninits=true]{biblatex}
\AtEveryBibitem{\clearfield{number}}
\AtEveryBibitem{\clearfield{eprintclass}}
\renewbibmacro{in:}{%
  \ifentrytype{article}
    {}
    {\bibstring{in}%
     \printunit{\intitlepunct}}}
     \DeclareFieldFormat[thesis]{title}{\textit{#1}}

\DeclareFieldFormat
  [article,inbook,incollection,inproceedings]
  {title}{#1}
  \DeclareFieldFormat
  [misc]
  {title}{#1,\textit{ Preprint}\nopunct}
  \DeclareFieldFormat
  [misc]
  {date}{\mkbibparens{#1},}
  \DeclareFieldFormat
  [article,inbook,incollection,inproceedings,patent,
   unpublished,techreport,misc]
  {volume}{\textbf{#1}}

\AtEveryBibitem{%
  \ifentrytype{online}{%
  }{%
    \clearfield{url}%
    \clearfield{urldate}%
  }%
}
\begin{document}
\global\long\def\A{\mathbb{A}}%

\global\long\def\C{\mathbb{C}}%

\global\long\def\E{\mathbb{E}}%

\global\long\def\F{\mathbb{F}}%

\global\long\def\G{\mathbb{G}}%

\global\long\def\H{\mathbb{H}}%

\global\long\def\N{\mathbb{N}}%

\global\long\def\P{\mathbb{P}}%

\global\long\def\Q{\mathbb{Q}}%

\global\long\def\R{\mathbb{R}}%

\global\long\def\O{\mathcal{O}}%

\global\long\def\Z{\mathbb{Z}}%

\global\long\def\ep{\varepsilon}%

\global\long\def\laurent#1{(\!(#1)\!)}%

\global\long\def\wangle#1{\left\langle #1\right\rangle }%

\global\long\def\ol#1{\overline{#1}}%

\global\long\def\mf#1{\mathfrak{#1}}%

\global\long\def\mc#1{\mathcal{#1}}%

\global\long\def\norm#1{\left\Vert #1\right\Vert }%

\global\long\def\et{\textup{ét}}%

\global\long\def\Et{\textup{Ét}}%

\title{Terminal singularities of the moduli space of curves on low degree hypersurfaces and the circle method}
\author{Jakob Glas and Matthew Hase-Liu}
\maketitle
\begin{abstract}
We study the singularities of the moduli space of degree $e$ maps from smooth genus $g$ curves to an arbitrary smooth hypersurface of low degree. For $e$ large compared to $g$, we show that these moduli spaces have at worst terminal singularities. Our main approach is to study the jet schemes of these moduli spaces by developing a suitable form of the circle method. 
\end{abstract}
\tableofcontents{}

\section{Introduction}
The interplay between geometry and number theory has had a long and rich history. Examples like Deligne's resolution of the Weil conjectures \cite{deligne1974conjecture, deligne1980conjecture}, which serves as a powerful tool for estimating exponential sums and makes crucial use of the heavy machinery of algebraic geometry, shows the impact of geometry on number theory should not be underestimated. 

Occasionally, this flow of information can be reversed and tools from analytic number theory can be used to investigate objects of geometric interest. For instance, building on ideas of Ellenberg and Venkatesh, Browning and Vishe~\cite{timpankaj} used the circle method over function fields to deduce crude geometric properties of the moduli space $\mathcal{M}_{0,0}(X,e)$ of rational curves of degree $e$ on a smooth hypersurface $X$ of low degree.

Our aim in this paper is to enrich this flow of information by developing a suitable form of the circle method to show the  moduli space $\mathcal{M}_{g,0}(X,e)$ of smooth projective genus $g$ curves of degree $e$ on $X\subset \P^n_{\C}$ has at worst terminal singularities for $n$ large compared to $d$ and $e$ large compared to $d$ and $g$.

Let 
\begin{equation}\label{Eq: ExpectedDim}
    \bar{\mu}=\underset{\dim H^0(\deg e \text{ line bundle on genus }g\text{ curve})^{n+1}}{\underbrace{\left(n+1\right)\left(e-g+1\right)}}-\underset{\dim H^0(\deg de \text{ line bundle})}{\underbrace{\left(de-g+1\right)}}+\underset{\dim\Pic^{e}\left(\text{genus }g\text{ curve}\right)}{\underbrace{g}}-\underset{\text{scaling}}{\underbrace{1}}+\underset{\dim\mathcal{M}_{g}}{\underbrace{\left(3g-3\right)}}.
\end{equation}
A naive heuristic based on Riemann--Roch and deformation theory suggests that $\dim \mathcal{M}_{g,0}(X,e)=\bar{\mu}$. For $g=0$, $n\geq d+2$, and $X$ \emph{general} in moduli, Riedl and Yang \cite{RiedlYang} have shown this is indeed the case, in addition to establishing irreducibility of $\mathcal{M}_{0,0}(X,e)$; in fact, they even demonstrate these results for the Kontsevich compactification $\ol{\mathcal{M}}_{0,0}(X,e)$. 

Methods based on analytic number theory from Browning and Vishe \cite{timpankaj} allow one to deal with \emph{any} smooth hypersurface at the cost of requiring $n\ge(5d-4)2^{d-1}$, which was later refined to $n\ge (2d-1)2^{d-1}$ by Browning and Sawin \cite{BrowningSawinFree}. Bilu, Browning, and Sawin \cite{cohomology_circle, bilu2023motivic} have also developed cohomological and motivic analogs of this strategy. For the higher genus setting, the second author \cite{haseliu2024higher} extended the above results on irreducibility and expected dimension to $\mathcal{M}_{g,0}(X,e)$ by re-interpreting the analytic strategy geometrically. Recent work of Sawin~\cite{sawin2024asymptoticwaringsproblemfunction} also addresses the case of $X$ a Fermat hypersurface, in which case the lower bound on $n$ can be improved to depend only linearly on $d$.

Perhaps the next natural geometric question is to ask how singular $\mathcal{M}_{g,0}(X,e)$ is. For $d=1,2$, smoothness of $\mathcal{M}_{0,0}(X,e)$ follows from work of Kontsevich~\cite{kontsevich}, where he works in the more general context of \emph{convex} varieties $X$. For higher genus, the second author and Amal Mattoo \cite{non-smooth} proved that $\mathcal{M}_{g,0}(X,e)$ is essentially never smooth for any smooth Fano hypersurface $X$ of degree at least 2.

For $g=0$, Starr \cite{starr2003} investigated the singularities and Kodaira dimension of $\ol{\mathcal{M}}_{0,0}(X,e)$, proving that for $e\geq 2$ and $d\geq 3$, the spaces $\ol{\mathcal{M}}_{0,0}(X,e)$, as well as their coarse spaces (by inversion--of--adjunction) have at worst canonical singularities if $n \ge d + e$ and $X$ is \emph{general}. He also computed the canonical divisor of $\ol{\mathcal{M}}_{0,0}(X,e)$, proving that it is big in many cases. Essentially nothing else seems to be known about the quality of the singularities of $\mathcal{M}_{0,0}(X,e)$, certainly much less those of $\mathcal{M}_{g,0}(X,e)$ when $g\geq 1$.

Our main result is that for $n$ large compared to $d$, the singularities of the moduli spaces $\mathcal{M}_{g,0}(X,e)$ are as mild as possible, at least from the perspective of the minimal model program. For an integer $g\geq 0$, define 
\[
f(g)=\begin{cases}
    0 &\text{if }g\in\{0,1\},\\
    (g+1)/2 &\text{if }g\geq 2.
\end{cases}
\]

\begin{thm}\label{Thm: Main.TermSing}
    Let $X\subset \P^n_{\C}$ be a smooth hypersurface of degree $d\geq 2$ and $g\geq 0$ be an integer. Then $\mathcal{M}_{g,0}(X,e)$ has at worst terminal singularities when 
    \[
    n+1>\begin{cases}
        2^{d-2}d(2d+1) & \text{if }g=0, e=1 \text{, and }d\geq 4,\\
        2^{d-1}(d-1)\left(1+\frac{d-1}{2d-1}+\frac{2(d-1)(de+1)}{d}\right) &\text{if }g=0, 2\leq e\leq d-1 \text{, and } d\ge 3,\\
        2^{d-2}(d-1)(4d^2-4d+3) -1&\text{if }g=0, d \le e \text{, and }d\geq 3,\\
        2^{d-2}(d-1)(4d^2-4d+3) &\text{if }g\geq 1\text{ and }d\geq 2,
    \end{cases}
    \]
    provided that $e>e_0$, whose value is given in Table~\ref{Table: TermSing}. 
\end{thm}

\begin{table}[h]\centering
\caption{The value of $e_0$ (for terminal singularities)}
\label{Table: TermSing}
\begin{tabular}{|l|l|l|}
\hline 
$g$ & $d$ & $e_{0}$\tabularnewline
\hline 
\hline 
$=0$ & $\ge3$ & $=0$\tabularnewline
\hline 
$\ge1$ & $\ge3$ & $=2^{d-2}\left(d-1\right)^{2}\left(4d^{2}g+\left(4d^{3}-8d^{2}+7d-3\right)f(g)+4d(g-2)-g+4\right)$\tabularnewline
 &  & $\qquad-1+(d+1)g+(d-1)^{2}f(g)$\tabularnewline
\hline 
$\ge1$ & $=2$ & $=29g+14f(g)-5$\tabularnewline
\hline 
\end{tabular}
\end{table}

By precisely the same methods, we have the following similar result for canonical singularities, where we can slightly improve the range of admissible number of variables.
\begin{thm}\label{Thm: Main.CanSing}
    Let $X\subset \P^n_{\C}$ be a smooth hypersurface of degree $d\geq 2$ and $g\geq 0$ be an integer. Then $\mathcal{M}_{g,0}(X,e)$ has at worst canonical singularities when 
    \[
    n+1>\begin{cases}
        2^{d-2}d(2d+1) &\text{if }g=0, e=1 \text{, and }d\geq 4,\\ 
        2^{d-1}(d-1)(de+2) & \text{if }g=0, 2\leq e\leq d-2\text{, and }d\geq 3,\\
        2^{d-1}(d-1)(d^2-d+1) -1&\text{if }g=0, d-1 \le e \text{, and }d\geq 3,\\
        2^{d-1}(d-1)(d^2-d+1) &\text{if }g\geq 1\text{ and }d\geq 2,
    \end{cases}
    \]
    provided $e>e_0$, whose value is given in Table~\ref{Table: CanSing}.
\end{thm}
\begin{table}[h] \centering
\caption{The value of $e_0$ (for canonical singularities)}
\label{Table: CanSing}
\begin{tabular}{|l|l|l|}
\hline 
$g$ & $d$ & $e_{0}$\tabularnewline
\hline 
\hline 
$=0$ & $\ge3$ & $=0$\tabularnewline
\hline 
$\ge1$ & $\ge3$ & $=(g-1)(2^{d-1}(d-1)^{2}(2d-1)+2)$\tabularnewline
 &  & $\qquad+(d-1)(g+(d-1)f(g))(2^{d-1}(d-1)(d^{2}-d+1)+1)$\tabularnewline
\hline 
$\ge1$ & $=2$ & $=19g+7f(g)-3$\tabularnewline
\hline 
\end{tabular}
\end{table}

For the particular case $e=1$, the space $F_1(X)=\mathcal{M}_{0,0}({X},{1})$ is the \emph{Fano variety of lines} of $X$. It is a classical result due to Altman and Kleiman~\cite{AltKleinFano} that the Fano variety of lines of \emph{any} smooth cubic hypersurface $X\subset \P^{n}$ is smooth for $n\geq 4$. If $X\subset \P^{n}$ is a hypersurface of degree $d\geq 4$, then $F_1(X)$ is still known to be smooth, if one assumes that $X$ is general and $d\leq 2n-3$ \cite[Theorem V.4.3]{Kollar_Rational_Curves}. In particular, Theorem \ref{Thm: Main.TermSing} gives new information about $F_1(X)$ for smooth $X$ as soon as $d\geq 4$ and $n+1>2^{d-2}d(2d+1)$. 

Under these circumstances, Theorem \ref{Thm: Main.TermSing} even gives new insight into the Hodge numbers of $F_1(X)$. Recall that for a singular projective variety $Y$ over $\C$, the cohomology groups $H^{p+q}(X,\Q)$ admit a mixed Hodge structure with an increasing Hodge filtration $F_\bullet$. Then, the Hodge numbers $h^{p,q}(Y)$ are defined as the $\C$-dimensions of the graded pieces $\Gr^F_{-p}H^{p+q}(X,\C)$. Combining our main results with those of Park and Popa~\cite{park2024lefschetztheoremsqfactorialityhodge} on the relationship between symmetries of the Hodge diamond of $Y$ and higher rational singularities of $Y$ and those of Bilu and Browning~\cite{bilu2023motivic} on the Hodge--Deligne polynomial of the Fano variety of lines gives the following corollary.
\begin{cor}
    Let $X\subset \P^n_{\C}$ be a smooth hypersurface of degree $d\ge 4$ with $n+1>2^{d-2}d(2d+1)$. Denote by $N=2n-d-3$ the dimension of $F_1(X)$ and $h^{p,q}=h^{p,q}(F_1(X))$ for $0\le p,q\le N$. 
   
    Then, $h^{0,q}=h^{q,0}=h^{N,N-q}=h^{N-q,N}$ for $0\le q\le N$, which is moreover 0 for $q>4n+2d-6-2^{2-d}(n+1)$. Also, $h^{1,1}=h^{N-1,N-1}=1$.
\end{cor}
\begin{proof}
    By Theorem \ref{Thm: Main.TermSing}, $F_1(X)$ is a normal complete intersection with at worst terminal singularities and hence at worst rational singularities. The claim $h^{0,q}=h^{q,0}=h^{N,N-q}=h^{N-q,N}$ for $0\le q\le N$ then follows immediately from Theorem C of \cite{park2024lefschetztheoremsqfactorialityhodge}. Corollary 1.6 of \cite{bilu2023motivic} then gives $h^{0,q}=h^{q,0}=0$ provided $q>4n+2d-6-2^{2-d}(n+1)$.
    
    Moreover, it follows from Theorem~1.2 of \cite{BrowningSawinFree} that the singular locus of $F_1(X)$ has codimension at least $n/2^{d-2}-6(d-1)\ge d(2d+1)-6(d-1)>3$, so by Corollaire 3.14 of \cite{grothendieck2005cohomologielocaledesfaisceaux}, $F_1(X)$ is factorial. By Theorem A of \cite{park2024lefschetztheoremsqfactorialityhodge}, since $F_1(X)$ is a factorial complete intersection with rational singularities, we have the equality of Betti numbers $h^2(F_1(X))=h^{2N-2}(F_1(X))$. But we have already shown that the identities $h^{0,2}=h^{2,0}=h^{N,N-2}=h^{N-2,N}$ hold, so we must have $h^{1,1}=h^{N-1,N-1}$. Corollary 1.6 of \cite{bilu2023motivic} gives $h^{N-1,N-1}=\lfloor1/2\rfloor+1=1$, since $N-1+N-1>4n+2d-6-2^{2-d}(n+1)$.
\end{proof}

Of course, instead of investigating the quality of the singularities, one may ask the related question of whether the singular locus of $\mathcal{M}_{g,0}(X,e)$ is large. Because terminality only appears in at least codimension three, our results give a weak lower bound on the codimension of the singular locus of $\mathcal{M}_{g,0}(X,e)$. When $d\geq 3$ and $n\geq 2d$, Harris, Roth and Starr \cite{HarrisRothStarrRatCurvesI} showed that $\mathcal{M}_{0,0}(X,e)$ is generically smooth for generic $X$. In an unpublished note, Starr and Tian moreover proved a lower bound for the codimension of the singular locus for general hypersurfaces of low degree using a bend-and-break approach. For an arbitrary smooth hypersurface $X$, Browning and Sawin \cite{BrowningSawinFree} used the circle method to give stronger bounds on the codimension of the singular locus of $\mathcal{M}_{0,0}(X,e)$ at the expense of an exponential lower bound on $n$ in terms of $d$. Recently, in \cite{codimjumping}, Lehmann, Riedl, and Tanimoto used ideas from developments related to geometric Manin's conjecture to give a non-explicit linear lower bound on the codimension of families of non-free curves that do not arise from accumulating maps for even the higher genus setting and an arbitrary Fano variety $X$.

\subsection*{Outline} We now give a brief overview of the main steps in the proofs of Theorems \ref{Thm: Main.TermSing} and \ref{Thm: Main.CanSing}. Instead of working with $\mathcal{M}_{g,0}(X,e)$ directly, it suffices to work with the moduli space $\Mor(C,X,L)$ of degree $e$ morphisms from a fixed curve $C$ to $X$ corresponding to a specific line bundle $L$. We explain this reduction in Section~\ref{reductnaive}.

More precisely, recall there is a natural map $\Mor_{e}\left(C,X\right)\to\Pic^{e}\left(C\right)$
given by pulling back $\mathcal{O}_{X}\left(1\right)$ to $C\times T$ along
$C\times T\to X\times T$ for a $T$-point of $\Mor_{e}\left(C,X\right)$. Denote by $\Mor(C,X,L)$ the fiber above a closed point $[L]\in \Pic^e(C)$.  Our proof relies on studying the jet schemes $J_m(\Mor(C,X,L))$ as well as the iterated jet schemes $J_1\left(J_m(\Mor(C,X,L))\right)$ for any integer $m\geq 0$. The link to terminal and canonical singularities is provided by a result due to Musta\c{t}\u{a} \cite{mustataJetLCI}, which relates terminality and canonicity of local complete intersection varieties to geometric properties about their (iterated) jet schemes, namely that for all $m$, the schemes $J_m(\Mor(C,X,L))$ and $J_1\left(J_m(\Mor(C,X,L))\right)$ are irreducible local complete intersection varieties. 

More precisely, we will deduce Theorems \ref{Thm: Main.TermSing} and \ref{Thm: Main.CanSing} from the following results. 
\begin{prop}\label{Prop: Main.Jet.Term}
     Let $C$ be a smooth projective curve of genus $g$ over the complex numbers. Under the assumptions of Theorem \ref{Thm: Main.TermSing}, for any $m\geq 0$, the jet schemes $J_m(\Mor(C,X,L))$ and iterated jet schemes $J_1(J_m(\Mor(C,X,L)))$ are irreducible and of the expected dimensions $(m+1)\dim\Mor(C,X,L)$ and $2(m+1)\dim\Mor(C,X,L)$, respectively, for all $[L]\in \Pic^e(C)$.
\end{prop}

\begin{prop}\label{Prop: Main.Jet.Can}
    Let $C$ be a smooth projective curve of genus $g$ over the complex numbers. Under the assumptions of Theorem \ref{Thm: Main.CanSing}, the jet schemes $J_m(\Mor(C,X,L))$ are irreducible of the expected dimension $(m+1)\dim\Mor(C,X,L)$ for all $m\ge 0$ and for all $[L]\in \Pic^e(C)$.
\end{prop}

\begin{rem}
    It seems plausible, though extremely tedious, to prove results similar to Propositions \ref{Prop: Main.Jet.Term} and \ref{Prop: Main.Jet.Can} for iterated jet schemes of the form $J_1(J_1(\cdots J_1(J_m(\Mor(C,X,L)))\cdots ))$ by the same strategy that we present in this paper. More precisely, if $u$ is the number of $J_1$'s in this iterated jet scheme, by appropriately modifying the proof of Theorem 3.3 in \cite{jet_inversion}, it seems likely that one should be able to show for $n$ sufficiently large compared to $d^u2^d$ and $e$ large compared to $d$, $u$, and $g$, the discrepancies of $\Mor(C,X,L)$ are all larger than $u$. 
    
    It would also be interesting to investigate the behavior of iterated higher-order jet schemes of the form $J_m(J_{m'}(\Mor(C,X,L)))$ for arbitrary $m,m'$, for which generalizations of Musta\c{t}\u{a}'s results relating jet schemes and canonical singularities are not known. See \cite{dimjetsing} for further questions and discussions about generalized jet schemes.  
\end{rem}

Let us now discuss our approach in more detail, restricting to the case $C=\P^1$ for simplicity. Our strategy follows the path paved by Browning and Vishe~\cite{timpankaj}. More precisely, after performing a spreading out argument, it suffices to prove the corresponding results over a finite field $\F_q$. By appealing to the Lang--Weil estimate this is equivalent to understanding the number of $\F_q$-points on $J_m(\Mor_e(C,X))$ and $J_1(J_m(\Mor_e(C,X)))$ as $q$ goes to infinity. At least for $C=\P^1$, these correspond to $\F_q[t][s]/s^{m+1}$-points and $\F_q[t][s,r]/\left\langle s^{m+1},r^2\right \rangle$-points of degree $e$ in $t$ on $x$. We solve this counting problem by developing a suitable version of the circle method. When $C=\P^1$, this counting problem can be interpreted as counting the number of $\F_q[t]$-points of bounded height on a system of $m+1$ equations (resp. $2(m+1)$) in $(n+1)(m+1)$ (resp. $2(n+1)(m+1)$) variables. Classical applications of the circle method require the number of variables to grow roughly quadratically in the number of equations. In particular, since for $n,d$ and $e$ we want to count points for \emph{all} $m\geq 0$, this approach seems hopeless and in fact Browning and Sawin~\cite{BrowningSawinFree} even wrote in their work that it does not seem possible to prove that $\mathcal{M}_{0,0}(X,e)$ has canonical singularities using the circle method. Nevertheless, we succeed using a modified version of the circle method, as we shall now explain. 

The key point is that we do not treat the equations for $J_m(\Mor_e(C,X))$ as a system of equations over $\F_q(t)$, but keep working with it as a single equation over $\F_q[t][s]/s^{m+1}$. Our situation may thus be compared to that over number fields: if one attempts to count solutions to a single equation over a number field, but uses Weil restriction to consider it as a system of equations over $\Q$, then a naive application of the circle method requires the number of variables to grow quadratically in the degree of the number field. However, as demonstrated by Skinner \cite{SkinnerForms}, if things are set up correctly, one can essentially still treat the system as a single equation and obtain completely analogous results to those of $\Q$. 

In our approach we use a geometric interpretation of harmonic analysis over $(\F_q[s]/s^{m+1})(\!(t^{-1})\!)$ following \cite{haseliu2024higher}, which is particularly useful in the higher genus setting. Then, we continue along the standard lines of attack of the circle method. This includes a division into minor and major arcs followed by a suitable form of Weyl differencing to bound the exponential sums that arise in this process.


The aforementioned results on irreducibility and expected dimension for moduli spaces of rational curves on smooth low degree hypersurfaces have been generalized to the setting of smooth complete intersections of the same degree by Browning, Vishe and Yamagishi~\cite{browning2024rationalcurvescompleteintersections}. More precisely, given ${p_1,\dots, p_b\in \P^1}$ and $y_1,\dots, y_b\in X$, they extended their method to deal with the spaces $\Mor_e(\P^1, X, p_1,\dots, p_b ; y_1,\dots, y_b)$ of degree $e$ morphisms $f\colon \P^1\to X$ satisfying $f(b_i)=p_i$ for $i=1,\dots, b$. Their approach is based on a function field version of Rydin~Myerson's work \cite{myerson2018quadratic, myerson2019systems} that allows one to handle systems of forms with the circle method for which the number of variables only grows linearly in the number of equations. It seems plausible that our work extends to the setting of \cite{browning2024rationalcurvescompleteintersections} and it would be interesting to see whether Rydin~Myerson's approach could be useful in the study of terminal singularities. 

\begin{ack}
    While working on this paper the first author was supported by an FWF grant (DOI 10.55776/P36278) and the second author was partially supported by an NSF Grant (DGE-2036197). The authors would like to thank Tim Browning, Aise Johan de Jong, Brian Lehmann, Amal Mattoo, Mircea Musta\c{t}\u{a}, Mihnea Popa, Eric Riedl, Will Sawin, and Sho Tanimoto for useful comments.
\end{ack}

\section{Preliminaries}
\subsection{Reduction to maps out of a fixed curve}\label{reductnaive}
In this subsection, we explain how to deduce a statement about terminal singularities of the moduli space $\mathcal{M}_{g,0}(X,e)$ from that of $\Mor(C,X,L)$.

\begin{lem}\label{Le: Sing.flatMorph}
Let $f\colon\mathcal{X}\to\mathcal{Y}$ be a flat and finite type morphism
of integral algebraic stacks locally essentially of finite type over
$\C$. Suppose $\mathcal{Y}$ is regular and that all closed fibers
of $f$ have at most terminal (resp. canonical) singularities. Then, $\mathcal{X}$ has at
worst terminal (resp. canonical) singularities. 
\end{lem}

\begin{proof}
This argument is due to Takumi Murayama (see \cite{takumi}). We first reduce to the case of schemes. Let $V\to\mathcal{Y}$ be
a smooth atlas and consider the fiber product $V\times_{\mathcal{Y}}\mathcal{X}$.
Then, let $U\to V\times_{\mathcal{Y}}\mathcal{X}$ be a smooth atlas
as well. Note that the composition $U\to V\times_{\mathcal{Y}}\mathcal{X}\to V$
satisfies the conditions of the lemma, since $V$ is regular over
$\C$, and the closed fibers, because they are smooth over the closed
fibers of $f$, have at worst terminal (resp. canonical) singularities by Proposition 2.15 of \cite{Kollar_2013} (having at worst terminal or canonical singularities is a smooth-local property). So we may assume $\mathcal{X}$ and $\mathcal{Y}$ are schemes.

Replacing $\mathcal{Y}$ with $\Spec\left(\mathcal{O}_{\mathcal{Y},y}\right)$
for $y\in\mathcal{Y}$, we can reduce to the case of $\mathcal{Y}$
the spectrum of a regular local ring, say with $x_{1},\ldots,x_{d}$
a regular system of parameters. 

For $d=0$, there is nothing to prove. Assume the claim is true inductively
for smaller $d$. By the induction hypothesis, $f^{-1}\left(V(x_{d})\right)$ has at
worst terminal (resp. canonical) singularities. Since $f^{-1}\left(V(x_{d})\right)$
is a Cartier divisor and small deformations of terminal (resp. canonical) singularities
are also terminal (resp. canonical) by Proposition 4.15 of \cite{Murayama_Grothendieck}, it follows that the singularities in a neighborhood
of $f^{-1}\left(V(x_{d})\right)$ are at worst terminal (resp. canonical). 

By Corollary 4.18(i) of \cite{Murayama_Grothendieck}, every fiber of $f$ has at worst terminal (resp. canonical) singularities. Then, the complement of $f^{-1}\left(V(x_{d})\right)$ also has at worst
terminal (resp. canonical) singularities by applying the induction hypothesis to the
restriction of $f$ to $\Spec\mathcal{O}_{\mathcal{Y},y}\left[x_{d}^{-1}\right]$,
since the latter has dimension less than $d$. 
The result follows.
\end{proof}

\begin{cor}\label{reduct_to_fixed}
Let $g\ge 0$ be an integer. For all smooth projective curves $C$ over $\mathbb{C}$ of genus $g$ and all $[L]\in \Pic^e(C)$, suppose $\Mor\left(C,X,L\right)$ is a non-empty local complete intersection of the expected dimension with at worst terminal (resp. canonical) singularities. Then, $\mathcal{M}_{g,0}\left(X,e\right)$ is also a local complete intersection with at worst terminal (resp. canonical) singularities.
\begin{proof}
Consider the forgetful morphism $f\colon\mathcal{M}_{g,0}\left(X,e\right)\to\mathcal{M}_{g}$, which on points forgets the map to $X$. We first claim that $f$ is flat.

To verify this, we invoke miracle flatness, i.e. checking that the source is Cohen--Macualay, the target is regular, and the fibers are of the same dimension.

In general, for a surjective map of varieties or Deligne--Mumford stacks $X\to Y$, we have $\dim X_y \ge \dim X - \dim Y$ for any point $y\in Y$. Since every fiber $\Mor_{e}\left(C,X\right)$ of $f$ has the expected dimension, the total space $\mathcal{M}_{g,0}\left(X,e\right)$ must also be of the expected dimension, i.e. it is a local complete intersection Deligne--Mumford stack cut out by $de-g+1$ equations in the smooth Deligne--Mumford stack $\mathcal{M}_{g,0}(\P^n,e)$ by Proposition 3 of \cite{non-smooth}. 

In particular, $\mathcal{M}_{g,0}\left(X,e\right)$ is Cohen--Macaulay, $\mathcal{M}_{g}$ is always smooth, and the fibers are of the same (expected) dimension, so we conclude that $f$ is flat.

The same argument shows that the map $\Mor_e(C,X) \to \Pic^e(C)$ is flat with closed fibers $\Mor(C,X,L)$. By assumption on the singularities of $\Mor(C,X,L)$, Lemma~\ref{Le: Sing.flatMorph} then implies that $\Mor_e(C,X)$ has at worst terminal (resp. canonical) singularities. Then, applying Lemma~\ref{Le: Sing.flatMorph} again to the map $f$ implies that $\mathcal{M}_{g,0}(X,e)$ has at worst terminal (resp. canonical) singularities. 
\end{proof}
\end{cor}

\begin{rem}
For $g=0$, we can argue more directly as follows: $\mathcal{M}_{0,0}\left(X,e\right)$
can be realized as the stack quotient $\left[\Mor_{e}\left(C,X\right)/\PGL_{2}\right]$,
and since $\PGL_{2}$ is smooth, we have $\Mor_{e}\left(C,X\right)\to\mathcal{M}_{0,0}\left(X,e\right)$
is an atlas, and hence we can simply use the fact that having at worst terminal (resp. canonical) singularities
is smooth-local.
\end{rem}

\subsection{Properties of jet schemes}
Following \cite{mustataJetLCI}, let us recall some basic facts about jet schemes. Fix a field of arbitrary characteristic $k$ and let $Y$ be a scheme of finite type over $k$. 

For each integer $m\ge 0$, the $m$th jet scheme $J_m(Y)$ is defined as the scheme whose representable functor of points $F$ in the category of finite type $k$-schemes satisfies
$$F(A) = Y\left(A[t]/t^{m+1}\right),$$
where $A$ is any finitely-generated $k$-algebra. For instance, there are canonical isomorphisms $J_0(Y)\cong Y$ and $J_1(Y)\cong \mathcal{T}Y$, the total space of the tangent bundle of $Y$.

The various jet schemes are equipped with canonical maps $\phi_m\colon J_m\left(Y\right) \to J_{m-1}\left(Y\right),$ which on the level of points are defined by modding out by $t^m.$ Let us write $\pi_m\colon J_m\left(Y\right) \to Y$ to denote the canonical projection (alternatively, this is simply the composition $\phi_1\circ \cdots \circ \phi_m$).

The construction of jet schemes is compatible with open immersions, and it is hence possible to describe them explicitly by reducing to the affine case. From this construction, one can show that if $Y$ is a smooth, connected variety of dimension $n$, then for every $m$, the morphism $\pi_m$ is an affine bundle whose fibers are $mn$-dimensional. Moreover, $J_m\left(Y\right)$ is smooth, connected, and $(m+1)n$-dimensional. 

In general, for any finite type $k$-scheme $X$, the \emph{expected dimension} of $J_m(X)$ is $(m+1)\dim X$.

In Theorem~1.3 of \cite{mustataInversionAdjunction} and Proposition~1.7 of \cite{mustataJetLCI}, Ein and Musta\c{t}\u{a} establishes the following criteria for having canonical and terminal singularities.
\begin{thm}
Let $Y$ be a local complete intersection variety. Then
\begin{enumerate}
    \item $Y$ has at worst terminal (resp. canonical) singularities
if and only if the jet schemes $J_{m}(Y)$ are normal (resp. irreducible) for every $m$, 
    \item $J_m(Y)$ is a local complete intersection variety if $J_m(Y)$ has the expected dimension, and 
    \item $Y$ is normal if $J_{m}(Y)$ is irreducible for some
$m\geq 1$ (equivalently for $m=1$).
\end{enumerate}

\end{thm}

Assuming the validity of Propositions \ref{Prop: Main.Jet.Term} and \ref{Prop: Main.Jet.Can}, Ein and Musta\c{t}\u{a}'s criteria implies the main theorems:

\begin{proof}[Proof of Theorems \ref{Thm: Main.TermSing} and \ref{Thm: Main.CanSing}, assuming Propositions \ref{Prop: Main.Jet.Term} and \ref{Prop: Main.Jet.Can}]
    Under the assumptions of Theorem \ref{Thm: Main.CanSing}, it follows that $J_{m}\left(\Mor \left(C,X,L\right)\right)$ is irreducible of the expected dimension for all $m$, which by Ein and Musta\c{t}\u{a}'s criteria implies that $\Mor\left(C,X,L\right)$ has at worst canonical singularities. Then, by Corollary \ref{reduct_to_fixed}, $\mathcal{M}_{g,0}(X,e)$ also has at worst canonical singularities.

    Similarly, under the assumptions of the Theorem \ref{Thm: Main.TermSing}, the schemes $J_{m}\left(\Mor\left(C,X,L\right)\right)$ and \linebreak $J_1\left(J_{m}\left(\Mor\left(C,X,L\right)\right)\right)$ are both irreducible of the expected dimension for all $m$, which by Ein and Musta\c{t}\u{a}'s criteria implies that $J_{m}\left(\Mor\left(C,X,L\right)\right)$ is normal for all $m$, which again by Ein and Musta\c{t}\u{a}'s criteria implies that $\Mor\left(C,X,L\right)$ has at worst terminal singularities. Then, by Corollary \ref{reduct_to_fixed}, $\mathcal{M}_{g,0}(X,e)$ also has at worst terminal singularities.
\end{proof}

So it remains to prove Propositions \ref{Prop: Main.Jet.Term} and \ref{Prop: Main.Jet.Can}.

\subsection{Spreading out and reduction to point-counting}

Let us recall the spreading out process as described in Section~2 of \cite{timpankaj}. For every $m$, note that both $J_m\left(\Mor\left(C,X,L\right)\right)$  and $J_1\left(J_m\left(\Mor\left(C,X,L\right)\right)\right)$ are finite type and hence can be defined over a finitely-generated $\Z$-algebra $\Lambda$ (by using the explicit equations cutting out these schemes). More precisely, for fixed $m$, we may assume $C,$ $X,$  $J_m\left(\Mor\left(C,X,L\right)\right)$,  and $J_1\left(J_m\left(\Mor\left(C,X,L\right)\right)\right)$, before base-changing to an algebraic closure, are the generic fibers of the finite type morphisms $\mathcal{C}\to \Spec \Lambda$, $\mathcal{X}\to \Spec \Lambda$, $Y_m\to \Spec \Lambda$, and $Y_{1,m}\to \Spec \Lambda$, respectively.

By Zariski's lemma, for any maximal ideal $\mathfrak{m}$ of $\Lambda$, the quotient $\Lambda/\mathfrak{m}$ is a finite field. Then, irreducibility of $J_m\left(\Mor\left(C,X,L\right)\right)$  and $J_1\left(J_m\left(\Mor\left(C,X,L\right)\right)\right)$ will follow from irreducibility of the fibers $Y_m\times \Spec \Lambda/\mathfrak{m}$ and $Y_{1,m}\times \Spec \Lambda/\mathfrak{m}$. 

Also, by possibly enlarging $\Lambda$, we may assume the fibers $\mathcal{C}\times \Spec \Lambda/\mathfrak{m}$ and $\mathcal{X} \times \Spec \Lambda / \mathfrak{m}$ are still smooth. By inverting $d!$ in $\Lambda$, we may also assume that the residue field of any maximal ideal has characteristic larger than $d$.

As established in the proof of Theorem 1 of \cite{haseliu2024higher} for $g\geq 1$ and Theorem 1.1 of \cite{BrowningSawinFree}, the space of morphisms $\Mor(C,X,L)$ for any $[L]\in\Pic^e(C)$ is irreducible and of the expected dimension, i.e. are local complete intersections, under the assumptions in either Proposition \ref{Prop: Main.Jet.Term} or \ref{Prop: Main.Jet.Can}. 

Using the affine bundle structure of the projection maps $\pi_m$ of jet schemes of the smooth loci, it follows that the jet schemes in consideration always have at least the expected dimension.
By Chevalley's theorem on constructible sets, there is a non-empty open subset $U\subset \Spec \Lambda$ such that the dimension of the generic fiber of $Y_m \to \Spec \Lambda$ (resp. $Y_{1,m} \to \Spec \Lambda$) is at most the dimension of any fiber above a closed point in $U$.
As a consequence, it suffices to show that $J_m\left(\Mor\left(C,X,L\right)\right)$  and $J_1\left(J_m\left(\Mor\left(C,X,L\right)\right)\right)$ are irreducible and of the expected dimension when $C$ and $X$ are defined over a finite field $\F_q$. 

By the Lang-Weil bounds, this is equivalent to the following:
\begin{equation}\label{eq: LW1}
\lim_{q\to\infty}q^{-(m+1)\mu}\#J_m\left(\Mor\left(C,X,L\right)\right)\left(\F_{q}\right)\le1
\end{equation}
and 
\begin{equation}\label{eq: LW2}
\lim_{q\to\infty}q^{-2(m+1)\mu}\#J_1\left(J_m\left(\Mor\left(C,X,L\right)\right)\right)\left(\F_{q}\right)\le1,
\end{equation}
where $\mu=\dim \Mor(C,X,L)=(n+1)(e-g+1)-(de-g+1) -1.$

More precisely, to prove Proposition \ref{Prop: Main.Jet.Term}, it suffices to establish \eqref{eq: LW1} and \eqref{eq: LW2} (under the assumptions of Proposition \ref{Prop: Main.Jet.Term}). Similarly, to prove Proposition \ref{Prop: Main.Jet.Can}, it suffices to establish \eqref{eq: LW1} (under the assumptions of Proposition \ref{Prop: Main.Jet.Can}).





Let $F\in \F_q[x_0,\dots, x_n]$ be a form of degree $d$ defining $X$ in $\P^n$.
The rational points of $\Mor\left(C,X,L\right)$ can be described explicitly as tuples of the following form:
\[
\left\{ \left(s_{0},\ldots,s_{n}\right)\in H^0(C,L)^{n+1}\colon F\left(s_{0},\ldots,s_{n}\right)=0,s_{0},\ldots,s_{n}\text{ globally generate}\right\} /\G_{m}(\F_q)
\]

Denote by $M\left(C,X,L\right)$ the cone of $\Mor\left(C,X,L\right),$ which we mean to denote the variety whose rational points are the same as those of $\Mor\left(C,X,L\right),$ except without the quotient by $\G_{m}.$

\begin{lem}\label{Le: RepsForJetSchemes}
An $\F_q[t]/t^{m+1}$-point of $M\left(C,X,L\right)$
is given by a tuple ${\vec{x}\in H^0(C,L)^{n+1}\otimes_{\F_q} \F_q[t]/t^{m+1}}$ such that $\vec{x}\bmod t$
globally generates and $F\left(\vec{x}\right)=0$. 

Similarly, $\F_q[t,r]/\wangle{t^{m+1}, r^2}$-points of $M\left(C,X,L\right)$ are given by  
$\vec{x}\in H^0(C,L)^{n+1}\otimes_{\F_q}\F_q[t,r]/\wangle{t^{m+1}, r^2}$ such that $\vec{x}\bmod \wangle{t,r}$
globally generates and $F\left(\vec{x}\right)=0$.
\begin{proof}
We will only prove the first claim; the second goes through identically. An $\F_q[t]/t^{m+1}$-point of $\Mor\left(C,X,L\right)$
is simply a morphism $C\otimes \F_q[t]/t^{m+1}\to X\otimes \F_q[t]/t^{m+1}$ with corresponding line bundle isomorphic to $L\otimes_{\F_q} \F_q[t]/t^{m+1}$,
which is a morphism $C\otimes \F_q[t]/t^{m+1}\to\P_{\F_q[t]/t^{m+1}}^{n}$
that factors through $X\otimes \F_q[t]/t^{m+1}$. 

This is given by an $(n+1)$-tuple of global sections $\vec{x}$ of
$H^{0}\left(C\otimes \F_q[t]/t^{m+1},L\otimes_k \F_q[t]/t^{m+1}\right)$, which is isomorphic to $H^{0}\left(C,L\right)\otimes \F_q[t]/t^{m+1}$, 
that globally generates and factors through $X\otimes \F_q[t]/t^{m+1}$. Suppose that $\vec{x}=(x_0,\dots, x_n)$. 
To globally generate means that the map of sheaves
\[
\left(x_{0},\ldots,x_{n}\right)\colon \mathcal{O}_{C}^{n+1}\otimes \F_q[t]/t^{m+1}\to L\otimes \F_q[t]/t^{m+1}
\]
is surjective. 

Suppose $\vec{x}\bmod t$ globally generates. Then, working on the level of stalks, we may assume $x_{0},\ldots,x_{n}\in R[t]/t^{m+1}$ where $R=\mathcal{O}_{C,p}$ at some point $p\in C$ such that $\wangle{x_{0}\bmod t,\ldots,x_{n}\bmod t}=R.$ Since $C$ is smooth, note that $R$ is a discrete valuation ring. 

This implies that we can find $a_{0},\ldots,a_{n}\in R$ such that
$a_{0}x_{0}+\cdots+a_{n}x_{n}=1+tr'$ for some $r'\in R$. However, $1+tr'$
is invertible in $R[t]/t^{m+1}$ and hence $\wangle{x_{0},\ldots,x_{n}}= R[t]/t^{m+1}$, as desired.
\end{proof}
\end{lem}

In particular, from the lemma, it follows that the rational points of $J_m\left(\Mor\left(C,X,L\right)\right)$ comprise tuples of the form
\[
\left\{ \left(s_{0},\ldots,s_{n}\right)\in \frac{ H^0(C,L)^{n+1}[t]}{\wangle{t^{m+1}}}\colon \begin{array}{l} F\left(s_{0},\ldots,s_{n}\right)=0, \\ s_{0},\ldots,s_{n}\text{ glob. gen. mod $t$}\end{array}\right\} /\G_{m}\left(\F_q[t]/t^{m+1}\right),
\]
and the rational points of $J_1\left(J_m\left(\Mor\left(C,X,L\right)\right)\right)$ comprise tuples of the form
\[
\left\{ \left(s_{0},\ldots,s_{n}\right)\in \frac{ H^0(C,L)^{n+1}[t,r]}{\wangle{t^{m+1},r^2}}\colon \begin{array}{l} F\left(s_{0},\ldots,s_{n}\right)=0, \\ s_{0},\ldots,s_{n}\text{ glob. gen. mod $t,r$}\end{array}\right\} /\G_{m}\left(\F_q[t,r]/\wangle{t^{m+1},r^2}\right),
\]

where $H^0(C,L)^{n+1}[t]/\wangle{t^{m+1}}$ is simply the tensor product $H^0(C,L)^{n+1}\otimes_{\F_q} \F_q[t]/t^{m+1}$ and likewise for $H^0(C,L)^{n+1}[t,r]/\wangle{t^{m+1}, r^2}$.

Then, in the same vein as for $M(C,X,L)$, for any $[L]\in \Pic^e(C)$, we denote by $M_m(C,X,L)$ and $M_{1,m}(C,X,L)$ the cones of $J_m(\Mor(C,X,L))$ and $J_1(J_m(\Mor(C,X,L)))$, respectively, i.e. whose rational points are the same except without the quotient by $J_m(\G_m)$, respectively $J_1(J_m(\G_m))$.


Finally, then, to verify inequalities \eqref{eq: LW1} and \eqref{eq: LW2}, it suffices to show the following:

\begin{equation}\label{eq: actualLW1}
    \lim_{q\to\infty}q^{-(m+1)(\mu+1)}\#M_m\left(C,X,L\right)\left(\F_{q}\right)\le1
\end{equation}
and 
\begin{equation}\label{eq: actualLW2}
\lim_{q\to\infty}q^{-2(m+1)(\mu+1)}\#M_{1,m}\left(C,X,L\right)\left(\F_{q}\right)\le1.    
\end{equation}

\section{\label{sec:settingupcircle}Setup of the circle method}

Let us recall from \cite{haseliu2024higher} the geometric interpretation of the circle method that will allow us to work more naturally in the setting of smooth projective curves with arbitrary genus.

Fix a finite field $\F_q$ and a line bundle $L$ on $C$ of degree $e$. Throughout this paper, we assume that $e \ge 2g-1$ so that $L$ is always non-special. For any non-negative integer $r$ and effective divisor $Z\subset C$, let us write $P_{re,C}$ to denote $H^0\left(C,L^{\otimes r}\right)$, $P_{re,Z}$ to denote $H^0\left(Z,L^{\otimes r}|_Z\right)$, and $P_{re-Z,C}$ to denote $H^0\left(C,L^{\otimes r}(-Z)\right)$. Moreover, if $m\ge 0$ is an integer, let $P_{re,C,m}$ denote $P_{re,C}\otimes \F_q[t]/t^{m+1}$, $P_{re, Z, m}$ denote $P_{re,C} \otimes \F_q[t]/t^{m+1}$, and $P_{re-Z,C,m}$ denote $P_{re-Z,C}\otimes \F_q[t]/t^{m+1}$.

The ``circle'' in the circle method is given by elements of $P_{de,C,m}^\vee\cong P_{de,C}^\vee \otimes \F_q[t]/t^{m+1}$. To relate this to the perspective of harmonic analysis over function fields, we refer the reader to Section~6 of \cite{haseliu2024higher}, which uses a Serre duality argument to show that linear functionals on $P_{de,C}$ are the same as a certain quotient of the completion of $\F_q(t)$ with respect to the infinite place. 

We say $\alpha \in P_{de,C}^\vee$ factors through an effective divisor $Z\subset C$ if $\alpha$ is determined by its restriction to $Z$, i.e. $\alpha$ lies in the image of the canonical map $P_{de,Z}^\vee \to P_{de,C}^\vee$. We denote this by $\alpha \sim Z$. For general $m\ge 0$, we say $\alpha\in P_{de,C,m}^\vee$ factors through $Z$ if the restriction of $\alpha$ to $P_{de,C}^\vee$ factors through $Z$, i.e. if $\alpha \bmod t$ factors through $Z$. Again, we write $\alpha \sim Z$. We also define $\deg \alpha$ to be the smallest integer $b$ for which there exists an effective divisor $Z$ of degree $b$ and $\alpha \sim Z$.

The analogue of Dirichlet's approximation theorem in the geometric setting is given by Lemma~9 of \cite{haseliu2024higher}:
\begin{lem}\label{Le: DirichletApprox}
    Any $\alpha\in P_{de,C}^\vee$ factors through an effective divisor of degree at most $de/2+1$.
\end{lem}

Let us write $\alpha \sim_{\min} Z$ if $\alpha$ factors minimally through an effective divisor $Z$.

Lemma 10 of \cite{haseliu2024higher} implies the following:

\begin{lem}\label{Le: MajorDisjoint}
    There is at most one effective divisor $Z$ such that $\alpha \sim_{\min} Z$ and $\deg Z \le e-2g+1$.
\end{lem}

Note that we can view the
form $F$ whose zero locus is the hypersurface $X$ as a function
\[
F\colon H^{0}\left(C,L\right)\otimes_{\F_{q}}\F_{q}[t]/t^{m+1}\to H^{0}\left(C,L^{\otimes d}\right)\otimes_{\F_{q}}\F_{q}[t]/t^{m+1}
\] given by the cup product, and the composition with $\alpha\in P_{de,C,m}^{\vee}$ is simply
\[
H^{0}\left(C,L\right)\otimes_{\F_{q}}\F_{q}[t]/t^{m+1}\to H^{0}\left(C,L^{\otimes d}\right)\otimes_{\F_{q}}\F_{q}[t]/t^{m+1}\to\F_{q}\otimes_{\F_{q}}\F_{q}[t]/t^{m+1}.
\]

Moreover, if $\alpha=\alpha_{0}+\alpha_{1}t+\cdots+\alpha_{m}t^{m}$
and $y_{0}+y_{1}t+\cdots+y_{m}t^{m}\in H^{0}\left(C,L^{\otimes d}\right)\otimes_{\F_{q}}\F_{q}[t]/t^{m+1}$,
then we define
\[
\alpha(y)\coloneqq\alpha_{0}\left(y_{0}\right)+\left(\alpha_{0}\left(y_{1}\right)+\alpha_{1}\left(y_{0}\right)\right)t+\cdots+\left(\alpha_{0}\left(y_{m}\right)+\cdots+\alpha_{m}\left(y_{0}\right)\right)t^{m}\in\F_{q}[t]/t^{m+1}.
\]

Finally, let us fix $\psi_{m}\colon\F_{q}[t]/t^{m+1}\to\C^{\times}$ to be the non-trivial
additive character given by 
\[
a_{0}+\cdots+a_{m}t^{m}\mapsto\prod_{i=0}^{m}\psi\left(a_{i}\right),
\]
where $\psi$ is a non-trivial additive character on $\F_{q}$. 
\subsection{Setup for $J_m(\Mor(C,X,L))$}

In the classical circle method, we express our counting question in terms of an integral of exponential sums, parametrized by the circle. Using the language of the geometric interpretation, for a given $\alpha \in P_{de,C,m}^\vee$, we define the associated exponential sum as follows. 

Let 
\begin{equation}\label{Eq: Def.S(alpha)}
S(\alpha)=\sum_{\substack{\vec{x}\in P_{e,C,m}^{n+1} \\ \vec{x}\text{ glob. gen.}}}\psi_{m}\left(\alpha\left(F(\vec{x})\right)\right).
\end{equation}

The condition ``$\vec x$ globally generates'' is shorthand for the condition ``$\vec x \bmod t$ globally generates $L$.'' 

Orthogonality of characters then gives \[\sum_{\alpha\in P_{de,C,m}^{\vee}}\psi_{m}\left(\alpha\left(F(\vec{x})\right)\right)\\
  =\begin{cases}
0 & \text{if }F(\vec{x})\ne0,\\
\#P_{de,C,m}^{\vee} & \text{else,}
\end{cases}\]
which, by switching sums, implies
\[
\sum_{\alpha\in P_{de,C,m}^{\vee}}S(\alpha) 
 =\sum_{\substack{\vec{x}\in P_{e,C,m}^{n+1} \\ \vec{x} \text{ glob. gen.}}}\sum_{\alpha\in P_{de,C,m}^{\vee}}\psi_{m}\left(\alpha\left(F(\vec{x})\right)\right)
  =\#P_{de,C,m}^{\vee}\#M_m\left(C,X,L\right)(\F_{q}).
\]

Rearranging and applying Riemann-Roch, we obtain
\[
\#M_m\left(C,X,L\right)(\F_{q})=q^{-\left(m+1\right)\left(de-g+1\right)}\sum_{\alpha\in P_{de,C,m}^{\vee}}S(\alpha),
\]
which, according to our goal \eqref{eq: actualLW1}, implies that it suffices to verify
\begin{equation}\label{Eq: Goal.ExpSums.Can}
q^{-(m+1)(n+1)(e+1-g)}\sum_{\alpha\in P_{de,C,m}^{\vee}}S(\alpha)\to1.    
\end{equation}

Let us refine this goal slightly. In the classical circle method, the circle is divided into ``major'' and ``minor'' arcs reflecting the expected size of the exponential sums. In particular, we will show the following.
\begin{prop}\label{Prop: MajorArcGoal.J1}
    Let $m\geq 1$. Then, 
    \[
    q^{-(m+1)(n+1)(e+1-g)}\sum_{\deg Z \leq e-2g+1}\sum_{\alpha \sim_{\text{min}}Z}S(\alpha)= q^{-m(n+1)(e+1-g)}\sum_{\alpha \in P_{de,C, m-1}^\vee}S(\alpha).
    \]
\end{prop}
\begin{prop}\label{Prop: MinorArcGoal.J1}
    Let $m\geq 1$ and suppose that $n,d,e$ and $g$ satisfy the assumptions of Theorem~\ref{Thm: Main.CanSing}. Then,
    \[
    q^{-(m+1)(n+1)(e+1-g)}\sum_{\deg \alpha > e-2g+1}S(\alpha)\to 0 \quad \text{as }q\to\infty.
    \]
\end{prop}
Motivated by this decomposition above, we say $\alpha \in P_{de,C, m}^\vee$ is \emph{major} if $\deg \alpha \leq e-2g+1$ and \emph{minor} otherwise. We will prove Propositions~\ref{Prop: MajorArcGoal.J1} and Propositions~\ref{Prop: MinorArcGoal.J1} in Sections~\ref{Sec: Major.M(m)} and \ref{Sec: PuttingEverythingTogether.Can}, respectively. Assuming the validity of the results, we can verify \eqref{eq: actualLW1} and hence complete the proof of Theorem~\ref{Thm: Main.CanSing}.
\begin{proof}[Proof of Theorem~\ref{Thm: Main.CanSing}]
    We will proceed by induction on $m$. The case $m=0$ holds by \cite{BrowningSawinFree} for $g=0$ and by \cite{haseliu2024higher} for $g \geq 1$ under even weaker assumptions on $n$ and $e$. So suppose $m\geq 1$ and that \eqref{Eq: Goal.ExpSums.Can} holds for $m-1$. We then have 
    \begin{align*}
        q^{-(m+1)(n+1)(e+1-g)}\sum_{\alpha\in P^\vee_{de,C,m}}S(\alpha) & =q^{-(m+1)(n+1)(e+1-g)}\left(\sum_{\deg Z \leq e-2g+1}\sum_{\alpha \sim_{\text{min}}Z}S(\alpha) +\sum_{\deg \alpha > e-2g+1}S(\alpha)\right),
    \end{align*}
    where we used Lemma~\ref{Le: MajorDisjoint} to show that any major $\alpha$ factors through a unique effective divisor. Proposition~\ref{Prop: MinorArcGoal.J1} shows that the contribution from $\deg \alpha > e-2g+1$ tends to $0$ as $q\to \infty$ and by Proposition~\ref{Prop: MajorArcGoal.J1} we have 
    \begin{align*}
        q^{-(m+1)(n+1)(e+1-g)}\sum_{\deg Z \leq e-2g+1}\sum_{\alpha \sim_\text{min}Z}S(\alpha) & = q^{-m(n+1)(e+1-g)}\sum_{\alpha \in P_{de,C, m-1}^\vee}S(\alpha) \to 1 
    \end{align*}
    as $q\to \infty$ by the induction hypothesis. 
\end{proof}




\subsection{Setup for $J_1(J_m(\Mor(C,X,L)))$}
The setup for $J_1(J_m(\Mor(C,X,L)))$ is completely analogous to that of $J_m(\Mor(C,X,L))$.
For given $\alpha,\beta\in P_{de,C,m}^{\vee}$, we define 
\begin{equation}\label{Eq: Def.S(alpha,beta)}
S\left(\alpha,\beta\right)=\sum_{\substack{\vec{x}_0, \vec{x}_1 \in P_{e,C,m}^{n+1} \\ \vec{x}_0 \text{ glob. gen.}}}\psi_{m}\left(\alpha\left(F\left(\vec{x}_0\right)+\beta\vec{x}_1\cdot\nabla F\left(\vec{x}_0\right)\right)\right).
\end{equation}

Here, the condition ``$\vec x$ globally generates'' is again shorthand for the condition ``$\vec x \bmod t$ globally generates $L$.''


By orthogonality of characters, we have 
\[\sum_{\alpha,\beta\in P_{de,C,m}^{\vee}}\psi_{m}\left(\alpha\left(F\left(\vec{x}_0\right)+\beta\vec{x}_1\cdot\nabla F\left(\vec{x}_0\right)\right)\right)=\begin{cases}
\left(\#P_{de,C,m}^{\vee}\right)^{2} & \text{if }F\left(\vec{x}_0\right)=\vec{x}_1\cdot\nabla F\left(\vec{x_{0}}\right)=0,\\
0 & \text{else,}
\end{cases}\]
which, by switching sums, implies  
\begin{align*}
\sum_{\alpha,\beta\in P_{de,C,m}^{\vee}}S\left(\alpha,\beta\right) 
&=\sum_{\substack{\vec{x}_0, \vec{x}_1 \in P_{e,C,m}^{n+1} \\ \vec{x}_0 \text{ glob. gen.}}}\begin{cases}
\left(\#P_{de,C,m}^{\vee}\right)^{2} & \text{if }F\left(\vec{x}_0\right)
=\vec{x}_1\cdot\nabla F\left(\vec{x}_0\right)=0,\\
0 & \text{else}
\end{cases}\\
  &=\left(\#P_{de,C,m}^{\vee}\right)^{2}\#M_{1,m}\left(C,X,L\right)(\F_q).
\end{align*}
To justify the second step, observe that if $\vec{x}\in H^0(C,L)[t,r]/\langle t^{m+1}, r^2\rangle$, we can uniquely write $\vec{x}=\vec{x}_0+r\vec{x}_1$ with $\vec{x}_0,\vec{x}_1\in P_{e,C, m}^{n+1}$. Then, $F(\vec{x})=0$ if and only if $F(\vec{x}_0)=\vec{x}_1\cdot\nabla F(\vec{x}_0)=0$. 

In particular, by Lemma~\ref{Le: RepsForJetSchemes} we have the identity
\[
\#M_{1,m}\left(C,X,L\right)(\F_q)=q^{-2\left(m+1\right)\left(de-g+1\right)}\sum_{\alpha,\beta\in P_{de,C,m}^{\vee}}S\left(\alpha,\beta\right).
\]
According to our goal \eqref{eq: actualLW2}, we want to show
\begin{equation}\label{Eq: Goal.J1Jm.Expsums}
q^{-2\left(m+1\right)(n+1)\left(e-g+1\right)}\sum_{\alpha,\beta\in P_{de,C,m}^{\vee}}S\left(\alpha,\beta\right)\to1    
\end{equation}
as $q\to\infty$. More precisely, we will show the following.
\begin{prop}\label{Prop: GoalMajorArc.J1(Jm)}
Let $m\geq 1$. Then,
\begin{align*}
\sum_{\deg Z\le e-2g+1}\sum_{\deg W\le e-2g+1}\sum_{\alpha\sim_{\min}Z}\sum_{\beta\sim_{\min}W}S\left(\alpha,\beta\right)=q^{2(n+1)(e+1-g)}\sum_{\alpha, \beta \in P_{de,C, m-1}^\vee}S(\alpha,\beta).
\end{align*}  
\end{prop}
\begin{prop}\label{Prop: GoalMinorArc.J1(Jm)}
Let $m\geq 1$ and suppose that $n,d,e$ and $g$ satisfy the assumptions of Theorem~\ref{Thm: Main.TermSing}. Then,
\[
q^{-2(m+1)(n+1)(e+1-g)}\sum_{\max\{\deg \alpha, \deg \beta\} > e-2g+1}S(\alpha,\beta)\to 0\quad\text{as }q\to\infty. 
\]
\end{prop}
As in the case of canonical singularities, if $\deg \alpha, \deg \beta \le e-2g+1$, we say the pair $(\alpha, \beta)$ is \emph{major}; else, we say the pair $(\alpha, \beta)$ is \emph{minor}. 
The proof of Propositions~\ref{Prop: GoalMajorArc.J1(Jm)} and \ref{Prop: GoalMinorArc.J1(Jm)} will be carried out in Sections~\ref{Sec: Major.J1(Jm)} and \ref{Sec: PuttingEverythingTogether.Term}, respectively. 

\begin{proof}[Proof of Theorem~\ref{Thm: Main.TermSing}]
It suffices to show that \eqref{Eq: Goal.J1Jm.Expsums} holds. We will proceed by an induction on $m$. Note that when $m=0$, then $J_1(J_0(M_m(C,X,L)))=J_1(M_m(C,X,L))$ implies that \eqref{Eq: Goal.J1Jm.Expsums} is the same as \eqref{Eq: Goal.ExpSums.Can} with $m=1$. The assumptions of Theorem~\ref{Thm: Main.CanSing} are certainly satisfied under the hypotheses of Theorem~\ref{Thm: Main.TermSing}, and so we may assume $m\geq 1$ from now on. 

We have 
\[
\sum_{\alpha,\beta \in P_{de,C, m}^\vee}S(\alpha,\beta)= \sum_{\deg W,\deg Z\leq e-2g+1}\sum_{\alpha\sim_{\min} Z}\sum_{\beta\sim_{\min} W}S(\alpha,\beta)+\sum_{\substack{\alpha,\beta \in P_{de,C,m}^\vee \\ \max(\deg \alpha,\deg\beta)> e-2g+1}}S(\alpha,\beta)
\]
By Proposition~\ref{Prop: MinorArcGoal.J1} the contribution from the minor $(\alpha,\beta)$ is $o(q^{2(m+1)(n+1)(e+1-g)})$ as $q\to\infty$, which is sufficient. In addition, Proposition~\ref{Prop: GoalMajorArc.J1(Jm)} implies 
\begin{align*}
    \sum_{\substack{\deg Z \leq e-2g+1 \\ \deg W\leq e-2g+1}}\sum_{\substack{\alpha \sim_{\min}Z \\ \beta\sim_{\min}W}}S(\alpha,\beta) & = q^{2(n+1)(e+1-g)}\sum_{\alpha,\beta \in P_{de,C, m-1}^\vee}S(\alpha,\beta)\\
    &= q^{2(m+1)(n+1)(e+1-g)}(1+o(1))
\end{align*}
as $q\to\infty$, where the last identity follows from the induction hypothesis.  
\end{proof}
\section{Major arcs}

\subsection{Major arcs for $J_m(\Mor(C,X,L))$}\label{Sec: Major.M(m)}

Let us first deal with the proof of Proposition~\ref{Prop: GoalMajorArc.J1(Jm)}.
Suppose $\alpha\sim_{\min}Z$ for $\deg Z\le e-2g+1$ and write $\alpha=\alpha_{0}+t\alpha_{1}$ with $\alpha_0\in P_{de,C}^\vee$ and $\alpha_1\in P_{de,C,m-1}^\vee$. For $\vec{x}\in P_{e,C,m}^{n+1}$, let us write $\vec{x}=\vec{y} + t^m \vec{z}$ with $\vec{y}\in P_{e,C,m-1}^{n+1}$ and $\vec{z}\in P_{e,C}^{n+1}$.

Then, using Taylor expansion we obtain
\begin{align*}
S(\alpha) & =\sum_{\substack{\vec{x}\in P_{e,C,m}^{n+1} \\ \vec{x}\text{ glob. gen.}}}\psi_{m}\left(\alpha\left(F(\vec{x})\right)\right)\\
 &= \sum_{\substack{\vec{y}\in P_{e,C,m-1}^{n+1} \\ \vec{y}\text{ glob. gen.}}} \sum_{\vec{z}\in P_{e,C}^{n+1}}\psi_{m}\left(\alpha\left(F\left(\vec{y}\right)+t^{m}\vec{z}\cdot\nabla F\left(\vec{y}\right)\right)\right)\\
 & =\sum_{\substack{\vec{y}\in P_{e,C,m-1}^{n+1} \\ \vec{y}\text{ glob. gen.}}}\psi_{m-1}\left(\alpha_{1}\left(F\left(\vec{y}\right)\right)\right)\psi_{m}\left(\alpha_{0}\left(F\left(\vec{y}\right)\right)\right)\sum_{\vec{z}\in P_{e,C}^{n+1}}\psi_{m}\left(\alpha_{0}\left(t^{m}\vec{z}\cdot\nabla F\left(\vec{y}\right)\right)\right).
\end{align*}
Let 
\[T\left(\alpha_{0}\right)
=\sum_{\vec{z}\in P_{e,C}^{n+1}}\psi_{m}\left(\alpha_{0}\left(t^{m}\vec{z}\cdot\nabla F\left(\vec{y}\right)\right)\right)\\
=\sum_{\vec{z}\in P_{e,C}^{n+1}}\psi\left(\alpha_{0}\left(\vec{z}\cdot\nabla F\left(\vec{y}\bmod t\right)\right)\right).\]
\begin{lem}\label{Lem: ExpSumVanish.Major.Jm}
    Let $\alpha \in P_{de,C,m}^\vee$ and suppose that $\alpha\sim_{\min}Z$ with $\deg Z\leq e-2g+1$. Then, with the notation above, we have 
    \[
T\left(\alpha_{0}\right)=\begin{cases}
q^{\left(n+1\right)\left(e-0-g+1\right)}q^{\left(n+1\right)0}=q^{\left(n+1\right)\left(e-g+1\right)} & \text{if }Z=0,\\
0 & \text{else}.
\end{cases}
\]
\end{lem}
\begin{proof}
Since $\alpha_{0}\sim Z$ and $\deg Z\leq e-2g+1$, using the exact sequence $0\to H^{0}\left(C,L\left(-Z\right)\right)\to H^{0}\left(C,L\right)\to H^{0}\left(Z,L|_{Z}\right)\to0$, and applying Riemann-Roch, we may write 
\[
T\left(\alpha_{0}\right)=q^{\left(n+1\right)\left(e-\deg Z-g+1\right)}\sum_{\vec{z}\in P_{e,Z}^{n+1}}\psi\left(\alpha_{0}\left(\vec{z}\cdot\nabla F\left(\vec{y}\bmod t|_{Z}\right)\right)\right),
\] viewing $\alpha_{0}$ as an element of $P_{de,Z}^{\vee}$. By orthogonality of characters we have 
\[
T(\alpha_0)=\begin{cases}
    q^{(n+1)(e-g+1)} &\text{if }\alpha_0(\vec{z}\cdot \nabla F(\vec{y} \bmod t_{|Z}))=0\text{ for all }\vec{z}\in P_{e,Z}^{n+1},\\
    0 &\text{else.}
\end{cases}
\]
Moreover, since $\alpha_{0}\sim_{\min}Z$, it follows that $\alpha_{0}\left(\vec{z}\cdot\nabla F\left(\vec{y}\bmod t|_{Z}\right)\right)=0$
identically in $\vec{z}$ if and only if $\alpha_{0}\left(z_{i}\frac{\partial F}{\partial y_{i}}\left(\vec{y}\bmod t|_{Z}\right)\right)=0$
for $0\leq i \leq n$, which holds if and only if $z_{i}\frac{\partial F}{\partial y_{i}}\left(\vec{y}\bmod t|_{Z}\right)=0$
for $0\leq i \leq n$. If this holds for every $\vec{z},$ then $\frac{\partial F}{\partial y_{i}}\left(\vec{y}\bmod t|_{Z}\right)=0$
for $0\leq i \leq n$. 

Since $F$ is smooth by assumption, the radical of the ideal $\wangle{\frac{\partial F}{\partial y_{0}}\left(\vec{y}\bmod t\right),\ldots,\frac{\partial F}{\partial y_{n}}\left(\vec{y}\bmod t\right)}$ is given by the ideal $\wangle{y_{0}\bmod t,\ldots,y_{n}\bmod t}$,
from which it follows that $y_{0}^{M}\bmod t,\ldots,y_{n}^{M}\bmod t$
are all in $\wangle{\frac{\partial F}{\partial y_{0}}\left(\vec{y}\bmod t\right),\ldots,\frac{\partial F}{\partial y_{n}}\left(\vec{y}\bmod t\right)}$
for sufficiently large $M$. Since $\frac{\partial F}{\partial y_{i}}\left(\vec{y}\bmod t|_{Z}\right)=0$
for $0\leq i\leq n$, it follows that $y_{0}^{M}\bmod t|_{Z},\ldots,y_{n}^{M}\bmod t|_{Z}=0$.
This means $MV\left(y_{0}\bmod t,\ldots,y_{n}\bmod t\right)-Z$ is
effective. However, since $y_{0} ,\ldots,y_{n}$ globally
generate modulo $t$, we must have $V\left(y_{0} \bmod t,\ldots,y_{n} \bmod t\right)=0$. Hence, $\alpha_{0}\left(\vec{z}\cdot\nabla F\left(\vec{y}\bmod t|_{Z}\right)\right)=0$ identically
if and only if $Z=0$. 
    
\end{proof}
Having Lemma~\ref{Lem: ExpSumVanish.Major.Jm} at our disposal, it is now straightforward to prove Proposition~\ref{Prop: MajorArcGoal.J1}.
\begin{proof}[Proof of Proposition~\ref{Prop: MajorArcGoal.J1}]
We have  
\begin{align*}
 & \sum_{\deg Z\le e-2g+1}\sum_{\alpha\sim_{\min}Z}S\left(\alpha\right)\\
 & =\sum_{\alpha_{1}\in P_{de,C,m-1}^{\vee}}\sum_{\substack{\vec{y}\in P_{e,C,m-1}^{n+1} \\ \vec{y} \text{ glob. gen.}}}\psi_{m-1}\left(\alpha_{1}\left(F\left(\vec{y}\right)\right)\right)\sum_{\deg Z\le e-2g+1}\sum_{\alpha_{0}\sim_{\min}Z}\psi_{m}\left(\alpha_{0}\left(F\left(\vec{y}\right)\right)\right)T\left(\alpha_{0}\right)\\
 & =\sum_{\alpha_{1}\in P_{de,C,m-1}^{\vee}}\sum_{\substack{\vec{y}\in P_{e,C,m-1}^{n+1} \\ \vec{y} \text{ glob. gen.}}}\psi_{m-1}\left(\alpha_{1}\left(F\left(\vec{y}\right)\right)\right)q^{(n+1)(e-g+1)}\\
 & =\sum_{\alpha_{1}\in P_{de,C,m-1}^{\vee}}S\left(\alpha_{1}\right)q^{(n+1)(e-g+1)},
\end{align*}
where the penultimate identity follows from Lemma~\ref{Lem: ExpSumVanish.Major.Jm} and the last from orthogonality of characters.
\end{proof}
\begin{rem}
    Lemma~\ref{Lem: ExpSumVanish.Major.Jm} is the reason why we keep working with globally generating tuples of sections in the definition of $S(\alpha,\beta)$ as opposed to other applications of the circle method. Indeed, if we drop this condition, then $T(\alpha_0)$ can be non-zero for non-trivial effective divisors $Z$. These would then give additional contributions in the proof of Proposition~\ref{Prop: MajorArcGoal.J1}, that seem to be hard to deal with. The same argument applies to the major arcs of $J_1(J_m(\Mor(C,X,L)))$.
\end{rem}

\subsection{Major arcs for $J_1(J_m(\Mor(C,X,L)))$}\label{Sec: Major.J1(Jm)}
We will now carry out the proof of Proposition~\ref{Prop: GoalMajorArc.J1(Jm)} in close parallel to Section~\ref{Sec: Major.M(m)}, with the key ingredient being again the vanishing result Lemma~\ref{Lem: ExpSumVanish.Major.Jm}.
\begin{proof}[Proof of Proposition~\ref{Prop: GoalMajorArc.J1(Jm)}]
    
Suppose that $\alpha\sim_{\min}Z$ and $\beta\sim_{\min}W$ for $\deg Z,\deg W\le e-2g+1$
and write $\alpha=\alpha_{0}+t\alpha_{1},\beta=\beta_{0}+t\beta_{1}$.
Also, write $\vec{x}_0=\vec{y}_0+t^{m}\vec{z}_0$
and $\vec{x}_1=\vec{y}_1+t^{m}\vec{z}_1$. 
Then, we obtain
\begin{align*}
 S\left(\alpha,\beta\right)
 &=\sum_{\substack{\vec{x}_0,\vec{x}_1\in P_{e,C,m}^{n+1} \\ \vec{x}_0\text{ glob. gen.}}}\psi_{m}\left(\alpha\left(F\left(\vec{x}_0\right)+\beta\vec{x}_{1}\cdot\nabla F\left(\vec{x}_{0}\right)\right)\right)\\
 & =\sum_{\substack{\vec{y}_{0}\in P_{e,C,m-1}^{n+1} \\ \vec{y}_{0}\text{ glob. gen.}}}\sum_{\vec{y}_{1}\in P_{e,C,m-1}^{n+1}}(\psi_{m-1}\left(\alpha_{1}F\left(\vec{y}_{0}\right)+\beta_{1}\vec{y}_{1}\cdot\nabla F\left(\vec{y}_{0}\right)\right)\psi_{m-1}\left(\alpha_{0}F\left(\vec{y}_{0}\right)\right)\\
 & \qquad\cdot\sum_{\vec{z}_{0}\in P_{e,C}^{n+1}}\psi\left(\alpha_{0}\vec{z}_{0}\cdot\nabla F\left(\vec{y}_{0}\right)\right)\psi_{m}\left(\beta_{0}\vec{y}_{1}\cdot\nabla F\left(\vec{y}_{0}+t^{m}\vec{z}_{0}\right)\right)\\
 & \qquad\cdot\sum_{\vec{z}_{1}\in P_{e,C}^{n+1}}\psi\left(\beta_{0}\vec{z}_{1}\cdot\nabla F\left(\vec{y}_{0}\right)\right)).
\end{align*}

By Lemma~\ref{Lem: ExpSumVanish.Major.Jm} we have
\[
\sum_{\vec{z}_{1}\in P_{e,C}^{n+1}}\psi\left(\beta_{0}\vec{z}_{1}\cdot\nabla F\left(\vec{y}_{0}\right)\right)=\begin{cases}
q^{\left(n+1\right)\left(e-0-g+1\right)}q^{\left(n+1\right)0}=q^{\left(n+1\right)\left(e-g+1\right)} & \text{if }W=0,\\
0 & \text{else},
\end{cases}
\]
and
\[
\sum_{\vec{z}_{0}\in P_{e,C}^{n+1}}\psi\left(\alpha_{0}\vec{z}_{0}\cdot\nabla F\left(\vec{y}_{0}\right)\right)=\begin{cases}
q^{\left(n+1\right)\left(e-0-g+1\right)}q^{\left(n+1\right)0}=q^{\left(n+1\right)\left(e-g+1\right)} & \text{if }Z=0,\\
0 & \text{else}.
\end{cases}
\]
It follows that the contribution of $W\ne0$ in \[\sum_{\deg Z\le e-2g+1}\sum_{\deg W\le e-2g+1}\sum_{\alpha\sim_{\min}Z}\sum_{\beta\sim_{\min}W}S\left(\alpha,\beta\right)\]
is 0, and the contribution of $Z\ne0$ in \[\sum_{\deg Z\le e-2g+1}\sum_{\alpha\sim_{\min}Z}\sum_{\beta\sim_{\min}0}S\left(\alpha,\beta\right)\]
is similarly 0. 

Then, we have 
\begin{align*}
 & \sum_{\deg Z\le e-2g+1}\sum_{\deg W\le e-2g+1}\sum_{\alpha\sim_{\min}Z}\sum_{\beta\sim_{\min}W}S\left(\alpha,\beta\right)\\
 & =\sum_{\alpha\sim_{\min}0}\sum_{\beta\sim_{\min}0}S\left(\alpha,\beta\right)\\
 & =\sum_{\alpha_{1},\beta_{1}\in P_{de,C,m-1}^{\vee}}\sum_{\substack{\vec{y}_{0}\in P_{e,C,m-1}^{n+1} \\ \vec{y}_{0} \text{ glob. gen.}}}\sum_{\vec{y}_{1}\in P_{e,C,m-1}^{n+1}}\psi_{m-1}\left(\alpha_{1}F\left(\vec{y}_{0}\right)+\beta_{1}\vec{y}_{1}\cdot\nabla F\left(\vec{y}_{0}\right)\right)\\
 & =\sum_{\alpha_{1},\beta_{1}\in P_{de,C,m-1}^{\vee}}S\left(\alpha_{1},\beta_{1}\right)q^{2(n+1)(e-g+1)},
\end{align*}
which is satisfactory.
\end{proof}

\section{Minor arcs}
In this section we establish upper bounds for the exponential sums $S(\alpha)$ and $S(\alpha,\beta)$  defined in \eqref{Eq: Def.S(alpha)} and \eqref{Eq: Def.S(alpha,beta)} via Weyl differencing, which will then be used to estimate the contribution from the minor arcs. 

As before, let $F\in \F_q[x_0,\dots, x_n]$ be a homogeneous degree $d$ form defining the smooth hypersurface $X\subset \P^n$, given by
\[
F(\vec{x})=\sum_{j_1,\dots, j_d=0}^n a_{j_1,\dots, j_d}x_{j_1}\cdots x_{j_d}
\]
where the $a_{j_1,\dots, j_d}\in\F_q$ are symmetric in the indices. A key role in our estimates is played by the multilinear forms 
\begin{equation}\label{Eq: Defi.MultilinForms}
\Psi_{j}\left(\vec{x}^{(1)},\ldots,\vec{x}^{(d-1)}\right)=d!\sum_{j_{1},\ldots,j_{d-1}=0}^{n}a_{j_{1},\ldots,j_{d-1},j}x_{j_{1}}^{(1)}\cdots x_{j_{d-1}}^{(d-1)},
\end{equation}
for $j=0,\dots, n$, where multiplication should be interpreted via the cup product. Given $\vec{y}\in P_{e,C,\ell}^{n+1}$ and a function $G\colon P_{e, C, \ell}^{n+1}\to P_{de,C, m}$ for some $\ell$, we define the differencing operator $D_{\vec{y}}$ via its action on $G$ by 
\[
D_{\vec{y}}(G)(\vec{x})=G(\vec{x}+\vec{y})-G(\vec{x}).
\]
For $k\geq 2$, we define recursively $D_{\vec{y}_1,\dots, \vec{y}_k}(G)=D_{\vec{y}_k}(D_{\vec{y}_1,\dots, \vec{y}_{k-1}}(G))$. The key property of the differencing operator is that if $G$ is a polynomial, then it lowers the degree of $G$. In the special case of a homogeneous form $F$ as above, we have for $G(\vec{x})=\vec{z}\cdot \nabla F(\vec{x})$ the identity
\begin{equation}\label{Eq: MultiLin+DiffOperator}
D_{\vec{y}^{(1)},\dots, \vec{y}^{(d-2)}}(G)(\vec{y}^{(d-1)})=\sum_{i=0}^{n} z_i \Psi_i(\vec{y}^{(1)},\dots, \vec{y}^{(d-1)}) 
\end{equation}
where $\Psi_0,\dots, \Psi_n$ are the multilinear forms associated to $F$ defined in \eqref{Eq: Defi.MultilinForms}. 
\subsection{Weyl differencing for $J_m(\Mor(C,X,L))$}\label{Se: WeylDiff.S(alpha)}
We begin with the treatment of $S(\alpha)$. If we write $m'=\lceil (m+1)/2\rceil$, then assuming $m\geq 1$ any $\vec{x}\in P_{e,C,m}^{n+1}$ can be written uniquely as $\vec{x}=\vec{y}+t^{m'}\vec{z}$ for some $\vec{y}\in P_{e,C, m'-1}^{n+1}$ and $\vec{z}\in P_{e,C, m-m'}^{n+1}$, so that we have $\psi_m(\alpha F(\vec{x}))=\psi(\alpha (F(\vec{y})+t^{m'} \vec{z}\cdot \nabla F(\vec{y})))$. Therefore, we have 
\[
S(\alpha)=\sum_{\vec{z}\in P^{n+1}_{e,C, m-m'}}T(\alpha),
\]
where 
\[
T(\alpha)=\sum_{\substack{\vec{y}\in P^{n+1}_{e,C, m'-1} \\ \vec{y}\text{ glob. gen.} }}\psi(\alpha((F(\vec{y})+t^{m'} \vec{z}\cdot \nabla F(\vec{y}))),
\]
where global generation of $\vec{y}$ means global generation after reducing modulo $t$. We now continue with the usual process of Weyl differencing. For $\vec{z}\in P_{e,C, m-m'}^\vee$, let $G'(\vec{y})=F(\vec{y})+t^{m'}\vec{z}\cdot \nabla F(\vec{y})$ and $G(\vec{y})=\vec{z}\cdot \nabla F(\vec{y})$. Squaring gives 
\[
|T(\alpha)|^2 = \sum_{\substack{\vec{y}_1\in P_{e,C, m'-1}^{n+1} \\ \vec{y}_1\text{ glob. gen.}}}\sum_{\substack{\vec{y}_2\in P_{e,C,m'-1}^{n+1} \\ \vec{y}_2\text{ glob. gen.}}}\psi(\alpha(G'(\vec{y}_1)-G'(\vec{y}_2))).
\]
After making the change of variables $\vec{y}^{(1)}=\vec{y}_2$ and $\vec{x}=\vec{y}_1-\vec{y}_2$, then the condition that $\vec{y}_1$ and $\vec{y}_2$ globally generate transforms into $\vec{y}^{(1)}$ and $\vec{y}^{(1)}+\vec{x}$ globally generate. We obtain 
\begin{align}
\begin{split}\label{Eq: IndStartWeylDiff}
|T(\alpha)|^2 &= \sum_{\substack{\vec{y}^{(1)}\in P_{e,C, m'-1}^{n+1}\\ \vec{y}^{(1)}\text{ glob. gen.}}}\sum_{\substack{\vec{x}\in P_{e,C, m'-1}^{n+1}\\ \vec{y}^{(1)}+\vec{x}\text{ glob. gen.}}}\psi(\alpha(D_{\vec{y}^{(1)}}(G')(\vec{x})))\\
&\leq \sum_{\substack{\vec{y}^{(1)}\in P_{e,C, m'-1}^{n+1}\\ \vec{y}^{(1)}\text{ glob. gen.}}}\left|\sum_{\substack{\vec{x}\in P_{e,C, m'-1}^{n+1}\\ \vec{y}^{(1)}+\vec{x}\text{ glob. gen.}}}\psi(\alpha(D_{\vec{y}^{(1)}}(G')(\vec{x})))\right| 
\end{split}
\end{align}

Inductively, an application of Cauchy--Schwarz yields
\begin{align*}
    |T(\alpha)|^{2^{d-2}}&\leq \left((\#P^{n+1}_{e,C, m'-1})^{2^{d-3}-(d-2)}\sum_{\vec{y}^{(1)}}\cdots \sum_{\vec{y}^{(d-3)}}\left|\sum_{\substack{\vec{x}\in P_{e,C,m'-1}^{n+1}\\ \vec{y}^{(1)}+\cdots + \vec{y}^{(d-3)}+\vec{x}\text{ glob. gen.}}}\psi(\alpha(D_{\vec{y}^{(1)},\dots, \vec{y}^{(d-3)}}(G')(\vec{x})))\right|\right)^2\\
    &\leq (\#P_{e,C, m'-1}^{n+1})^{2^{d-2}-(d-1)}\sum_{\vec{y}^{(1)}}\cdots \sum_{\vec{y}^{(d-3)}}\left|\sum_{\substack{\vec{x}\in P_{e,C,m'-1}^{n+1}\\ \vec{y}^{(1)}+\cdots + \vec{y}^{(d-3)}+\vec{x}\text{ glob. gen.}}}\psi(\alpha(D_{\vec{y}^{(1)},\dots, \vec{y}^{(d-3)}}(G')(\vec{x})))\right|^2\\
    &\leq (\#P_{e,C,m'-1}^{n+1})^{2^{d-2}-(d-1)}\sum_{\vec{y}^{(1)}}\cdots \sum_{\vec{y}^{(d-2)}}\sum_{\vec{y}^{(d-1)}}\psi(\alpha(D_{\vec{y}^{(1)},\dots, \vec{y}^{(d-2)}}(G')(\vec{y}^{(d-1)}))),
\end{align*}
where for each $1\leq i \leq d-1$, the sum over $\vec{y}^{(i)}\in P_{e,C, m'-1}^{n+1}$ is restricted to those such that $y^{(1)}+\cdots +\vec{y}^{(i)}$ globally generates and the last inequality follows from a change of variables as in \eqref{Eq: IndStartWeylDiff}.   
Note that $D_{\vec{y}^{(1)},\dots, \vec{y}^{(d-2)}}(G')(\vec{y}^{(d-1)})=D_{\vec{y}^{(1)},\dots, \vec{y}^{(d-2)}}(F)(\vec{y}^{(d-1)})+t^{m'}D_{\vec{y}^{(1)},\dots, \vec{y}^{(d-2)}}(G)(\vec{y}^{(d-1)})$. Applying Cauchy--Schwarz $d-2$ times gives 
\begin{equation}\label{Eq: CSMinorJm}
\left|S(\alpha)\right|^{2^{d-2}}\leq (\#P^{n+1}_{e,C,m-m'})^{2^{d-2}-1}\sum_{\vec{z}\in P^{n+1}_{e,C, m-m'}}\left|T(\alpha)\right |^{2^{d-2}}.    
\end{equation}
As $D_{\vec{y}^{(1)},\dots, \vec{y}^{(d-2)}}(F)(\vec{y}^{(d-1)})$ is clearly independent of $\vec{z}$ and recalling \eqref{Eq: MultiLin+DiffOperator}, interchanging the order of summation and taking absolute values yields
\begin{align*}
\left|S(\alpha)\right|^{2^{d-2}}&\leq (\#P_{e,C,m-m'}^{n+1})^{2^{d-2}-1}(\#P_{e,C,m'-1}^{n+1})^{2^{d-2}-(d-1)}\\
& \phantom{aa} \times \sum_{\vec{y}^{(1)}}\cdots \sum_{\vec{y}^{(d-1)}}\left|\sum_{\vec{z}\in P_{e,C,m-m'}^{n+1}}\psi(\alpha(t^{m'}\sum_{i=0}^n z_i\Psi_i(\vec{y}^{(1)},\dots, \vec{y}^{(d-1)})))\right|\\
&\leq (\#P^{n+1}_{e,C, m-m'})^{2^{d-2}}(\#P^{n+1}_{e,C, m'-1})^{2^{d-2}-(d-1)}N^{(m'-1,m+1-m')}(\alpha),    
\end{align*}
where for $k_1\geq  k_2-1$ we have defined
\[
N^{(k_1,k_2)}(\alpha)=\#\left\{(\vec{y}^{(1)},\dots, \vec{y}^{(d-1)})\in P_{e,C, k_1}^{(n+1)(d-1)}\colon \begin{array}{l} \alpha(\Psi_i(\vec{y}^{(1)},\dots, \vec{y}^{(d-1)})y)\equiv 0 \bmod s^{k_2} \\ \text{for all }0\leq i \leq n\text{ and }y\in P_{e,C, k_2-1}\end{array}\right\}.
\]
Note that here we dropped the condition about globally generating, which makes the counting function only larger. For any $k_1\geq k_2-1$, we have $N^{(k_1,k_2)}(\alpha)=\#P^{(n+1)(d-1)}_{e,C, k_1-k_2}N^{(k_2-1,k_2)}(\alpha)$ with the convention that $P_{e,C, -1}=\{0\}$, so that for $k\in \Z_{\geq -1}$ we shall write $N^{(k)}(\alpha)=N^{(k,k+1)}(\alpha)$ from now on. As $\#P_{e,C,m-m'}\#P_{e,C,m'-1}=\#P_{e,C,m}$ and $\#P_{e,C, m'-1}^{-1}\#P_{e,C,2m'-m-2}=\#P_{e,C,m-m'}^{-1}$, we have the proved the following result.
\begin{lem}\label{Le: WeylDiff.S(alpha)}
Let $\alpha \in P^\vee_{de,C,m}$ and suppose that $m\geq 1$. Then, 
\[
|S(\alpha)|^{2^{d-2}}\leq (\#P^{n+1}_{e,C, m})^{2^{d-2}} (\#P_{e,C, m-m'}^{n+1})^{-(d-1)}N^{(m-m')}(\alpha). 
\]
\end{lem}
\begin{rem}
    It might seem more natural to apply Weyl differencing to $S(\alpha)$ directly rather than to the sum $T(\alpha)$. The reason for not doing so is that we are working with globally generating sections in the definition of $S(\alpha)$. This would then cause problems when evaluating the resulting linear exponential sum using orthogonality of characters. 
\end{rem}
\subsection{Weyl differencing for $J_1(J_m(\Mor(C,X,L)))$}\label{Se: WeylDiff.S(alpha,beta)}
Next, we begin with the analysis of $S(\alpha,\beta)$. 
For any $\vec{z}\in P_{e,C, m}^{n+1}$, the combination of Cauchy--Schwarz and Weyl differencing arguments from the previous subsection can also be applied to the sum 
\[
S'(\alpha,\beta)=\sum_{\vec{y}\in P_{e,C, m}^{n+1}}\psi_m(\alpha F(\vec{y})+\beta \vec{z}\cdot \nabla F(\vec{y})),
\]
where the summation is restricted to elements that globally generate $\bmod~t$. Indeed, the final outcome of Weyl differencing only depends on the top degree part of $\alpha F(\vec{y})+\beta \vec{z}\cdot \nabla F(\vec{y})$ as a polynomial in $\vec{y}$, which after summing trivially over $\vec{z}$ thus leads to the estimate 
\begin{equation}\label{Eq: EstimateS(alpha,beta).alpha}
    |S(\alpha,\beta)|^{2^{d-2}}\leq (\#P^{2(n+1)}_{e,C, m})^{2^{d-2}} (\#P_{e,C, m-m'}^{n+1})^{-(d-1)}N^{(m-m')}(\alpha).
\end{equation}
Alternatively, we can first apply Weyl differencing to the sum over $\vec{y}$ only $d-2$ times and then bring the sum over $\vec{z}$ inside. Beginning with $d-2$ applications of Cauchy--Schwarz yields 
\begin{equation}\label{Eq: CSJ1Jm}
|S(\alpha,\beta)|^{2^{d-2}}\leq \#(P^{n+1}_{e,C, m})^{2^{d-2}-1}\sum_{\vec{z}\in P^{n+1}_{e, C, m}}|S'(\alpha,\beta)|^{2^{d-2}}.
\end{equation}
Arguing as in the lines after \eqref{Eq: CSMinorJm}, we obtain 
\[
|S'(\alpha,\beta)|^{2^{d-2}}\leq \#(P^{n+1}_{e,C, m})^{2^{d-2}-(d-1)}\sum_{\vec{y}^{(1)}}\cdots \sum_{\vec{y}^{(d-1)}}\psi_m(D_{\vec{y}^{(1)},\dots, \vec{y}^{(d-2)}}(\alpha F+\beta \vec{z}\cdot \nabla F))(\vec{y}^{(d-1)}).
\]
We can now insert this estimate into \eqref{Eq: CSJ1Jm} and take absolute values to get 
\begin{align}
\begin{split}\label{Eq: WeylDiff.S(alpha,beta).beta}
    |S(\alpha,\beta)|^{2^{d-2}} &\leq \# (P^{n+1}_{e,C,m})^{2^{d-1}-d}\sum_{\vec{y}^{(1)}}\cdots \sum_{\vec{y}^{(d-1)}}\sum_{\vec{z}}\psi_m(\beta \sum_{i=0}^n z_i \Psi_i(\vec{y}^{(1)},\dots, \vec{y}^{(d-1)}))\\
    &= \#(P^{n+1}_{e,C,m})^{2^{d-1}-d+1}N^{(m)}(\beta),
    \end{split}
\end{align}
where we dropped the condition on the $y^{(i)}$ about global generation of a suitable translate. Combining \eqref{Eq: EstimateS(alpha,beta).alpha} with \eqref{Eq: WeylDiff.S(alpha,beta).beta} and observing that $\#P_{e,C,m}=\#P_{e,C,m-m'}\#P_{e,C, m'-1}$, we arrive at the following result.
\begin{lem}\label{Le: WeylDiff.S(alpha,beta)}
    Let $\alpha,\beta \in P_{de,C,m}^\vee$ and suppose that $m\geq 1$. Then, 
    \[
    |S(\alpha,\beta)|^{2^{d-2}}\leq (\#P^{2(n+1)}_{e,C, m})^{2^{d-2}}(\#P^{n+1}_{e,C, m-m'})^{-(d-1)}\min\left\{N^{(m-m')}(\alpha), (\#P^{n+1}_{e,C, m'-1})^{-(d-1)}N^{(m)}(\beta)\right\}.
    \]
\end{lem}
\begin{rem}
    An alternative approach to estimate $S(\alpha,\beta)$ is to treat $F(\vec{x})=\vec{y}\cdot \nabla F(\vec{x})=0$ as a complete intersection of two forms and then use the standard approach in the circle method. Ignoring that globally generating sections cause difficulties here, that is to say, via Weyl differencing one can relate $S(\alpha,\beta)$ to certain multilinear forms associated to the pencil generated by $F(\vec{x})$ and $ \vec{y}\cdot \nabla F(\vec{x})$. However, our approach is more efficient and exploits the special shape of the equations with the result that we save one iteration of differencing. In the case $m=0$, this is explained in more detail in the paragraph after (5.3) in \cite{BrowningSawinFree}.
\end{rem}
\subsection{Shrinking lemma}
Our next goal is to establish a suitable form of Davenport's shrinking lemma. This will be achieved by reducing to the case $m=0$. We recall the following (slightly more general) result from Proposition 22 of \cite{haseliu2024higher}:

\begin{lem}\label{Le: GeneralShrinking.m=0}
Let $L_{0},\ldots,L_{M}$ be linear forms in $n+1$ variables $x_{0},\ldots,x_{n}$
whose coefficients belong to $P_{(e-s)(\ell-1)+e(d-\ell-1),C}$. Also, let
$\alpha_{0},\ldots,\alpha_{M'}\in P_{de,C}^{\vee}$ and $I_{0},\ldots,I_{n}$
be subsets of $\{0,\ldots,M\}\times\{0,\ldots,M'\}$. Then, for $\vec{x}\in P_{e,C}^{n+1}$,
we can think of $\sum_{(i,j)\in I_{k}}\alpha_{i}\left(L_{j}\left(\vec{x}\right)(-)\right)$
as a linear functional $P_{e+(\ell-1)s,C}^{\vee}$; for $\vec{x}\in P_{e-s,C}^{n+1}$,
we can think of $\sum_{(i,j)\in I_{k}}\alpha_{i}\left(L_{j}\left(\vec{x}\right)(-)\right)$
as a linear functional $P_{e+\ell s,C}^{\vee}$. 

Write 

\[
K_{\ell-1}\left(\alpha\right)=\#\left\{ \vec{x}\in P_{e,C}^{n+1}:\sum_{(i,j)\in I_{k}}\alpha_{i}\left(L_{j}\left(\vec{x}\right)(-)\right)=0\in P_{e+(\ell-1)s,C}^{\vee}\text{ for all }k\right\} 
\]
and 
\[
K'_{\ell}\left(\alpha\right)=\#\left\{ \vec{x}\in P_{e-s,C}^{n+1}:\sum_{(i,j)\in I_{k}}\alpha_{i}\left(L_{j}\left(\vec{x}\right)(-)\right)=0\in P_{e+\ell s,C}^{\vee}\text{ for all }k\right\} .
\]
Then, 
\[
\frac{K_{\ell-1}(\alpha)}{K'_{\ell}(\alpha)}\le\begin{cases}
q^{\left(n+1\right)s} & g=0\text{ and }s\ge0,\\
q^{\left(n+1\right)s} & g=1\text{ and }s\ge2,\\
q^{\left(n+1\right)s} & g\ge2,s\ge\max\left(2g-1,2\right),\text{ and }\ell\ge2,\\
q^{\left(n+1\right)\left(s+\left(g+1\right)/2\right)} & g\ge2,s\ge\max\left(2g-1,2\right),\text{ and }\ell=1.
\end{cases}
\]
\end{lem}

\begin{proof}
The proof is essentially the same as that of Proposition 22 in \cite{haseliu2024higher}, which in turn is a geometric reinterpretation of the classical shrinking lemma. 

The idea is to construct a certain vector bundle $E$ over $C$ (which should be understood as the geometric analogue of a lattice and choice of norm) such that the cardinalitites of $H^0\left(C,E\right)$ and $H^0\left(C,E\left(-s\infty\right)\right)$ are precisely $K_{\ell-1}(\alpha)$ and $K_\ell'(\alpha)$, respectively. 

In Section 7 of \cite{haseliu2024higher}, the construction of the vector bundle $E$ works identically if ``$\alpha\gamma$'' in the definition of the transition matrix $\Lambda$ is replaced with the linear combination $\sum_{(i,j)\in I_{k}}\alpha_{i}\left(L_{j}\left(\vec{x}\right)\right)$.

Then, we can bound $\log_{q}\left(K_{\ell-1}(\alpha)/K_\ell'(\alpha)\right)$ by bounding $h^0\left(C, E\right) - h^0\left(C,E\left(-s\infty\right)\right)$. This, in turn, can be achieved by using Riemann-Roch for vector bundles and the slopes of the Harder-Narasimhan filtration of $E$ (the latter should be viewed as successive minima of the corresponding lattice). In particular, the rest of the argument in Section 7 (and Proposition 22) of \cite{haseliu2024higher} goes through identically.

\end{proof}

It is now straightforward but notationally somewhat cumbersome to generalize Lemma~\ref{Le: GeneralShrinking.m=0} to $P_{e,C, k}$ when $k\geq 1$. 
\begin{lem}\label{Le: GeneralShrinking.m>0}
Let $L_{0},\ldots,L_{n}$ be linear forms in $n+1$ variables $x_{0},\ldots,x_{n}$
whose coefficients belong to $P_{(e-s)(\ell-1)+e(d-\ell-1),C,k}$. Also, let
$\alpha\in P_{de,C,k}^{\vee}$. Then, for $\vec{x}\in P_{e,C,k}^{n+1}$,
we can think of $\alpha\left(L_{j}\left(\vec{x}\right)(-)\right)$
as a linear functional $P_{e+(\ell-1)s,C,k}^{\vee}$; for $\vec{x}\in P_{e-s,C,k}^{n+1}$,
we can think of $\alpha\left(L_{j}\left(\vec{x}\right)(-)\right)$
as a linear functional $P_{e+\ell s,C,k}^{\vee}$. 

Write 
\[
K^{(k)}_{\ell-1}\left(\alpha\right)=\#\left\{ \vec{x}\in P_{e,C,k}^{n+1}:\alpha L_{j}\left(\vec{x}\right)=0\in P_{e+(\ell-1)s,C,k}^{\vee}\text{ for all }j\right\} 
\]
and 
\[
K^{(k)}_{\ell}\left(\alpha\right)=\#\left\{ \vec{x}\in P_{e-s,C,k}^{n+1}:\alpha L_{j}\left(\vec{x}\right)=0\in P_{e+\ell s,C,k}^{\vee}\text{ for all }j\right\} .
\]
Then, 
\[
\frac{K^{(k)}_{\ell-1}(\alpha)}{K^{(k)}_{\ell}(\alpha)}\le\begin{cases}
q^{(k+1)\left(n+1\right)s} & g=0\text{ and }s\ge0,\\
q^{(k+1)\left(n+1\right)s} & g=1\text{ and }s\ge2,\\
q^{(k+1)\left(n+1\right)s} & g\ge2,s\ge\max\left(2g-1,2\right),\text{ and }\ell\ge2,\\
q^{(k+1)\left(n+1\right)\left(s+\left(g+1\right)/2\right)} & g\ge2,s\ge\max\left(2g-1,2\right),\text{ and }\ell=1.
\end{cases}
\]
\begin{proof}
We can write 
\[
L_{i}\left(\vec{x}\right)=\sum_{j=0}^{k}t^{j}L_{i,j}\left(\vec{x}\right)
\]
for $0\le i\le n$, where each $L_{i,j}$ is a linear form in the $n+1$
variables $x_{0},\ldots,x_{n}$ with coefficients in $P_{(e-s)(\ell-1)+e(d-\ell-1),C}$. 

Write $\vec{x}=\vec{x}_{0}+t\vec{x}_{1}+\cdots+t^{k}\vec{x}_{k}$
with each $\vec{x}_{i}\in P_{e,C}^{n+1}$ for all $i$ or each $\vec{x}_{i}\in P_{e-s,C}^{n+1}$
for all $i$. Let 
\[
\widetilde{L}_{i,k}\left(\vec{x}_{0},\ldots,\vec{x}_{k}\right)=\sum_{j+\ell=k}L_{i,j}\left(\vec{x}_{\ell}\right),
\]
so that each $\widetilde{L}_{i,k}$ is a linear form in $\left(k+1\right)n$
variables. We can then write 
\[
L_{i}\left(\vec{x}\right)=\sum_{k=0}^{k}t^{k}\widetilde{L}_{i,k}\left(\vec{x}_{0},\ldots,\vec{x}_{k}\right).
\]
Let us also write $\alpha=\alpha_{0}+t\alpha_{1}+\cdots+t^{k}\alpha_{k}$,
so that we have 
\[
\alpha L_{i}\left(\vec{x}\right)=\left(\alpha_{0}\widetilde{L}_{i,0}\right)+t\left(\alpha_{0}\widetilde{L}_{i,1}+\alpha_{1}\widetilde{L}_{i,0}\right)+\cdots+t^{k}\left(\alpha_{0}\widetilde{L}_{i,k}+\cdots+\alpha_{k}\widetilde{L}_{i,0}\right).
\]
and it follows for each $i$ that $\alpha L_{i}\left(\vec{x}\right)=0\in P_{e+(\ell-1)s,C,k}^{\vee}$
if and only if $\alpha_0\tilde{L}_{i,j}+\cdots + \alpha_j \tilde{L}_{i,0}=0$ for $j=0,\dots, k$. Thus we can re-write $K_{\ell-1}\left(\alpha\right)$ and
$K_{\ell}\left(\alpha\right)$ as counting problems for which $k=0$ in
$(k+1)(n+1)$ variables. The result then follows from Lemma~\ref{Le: GeneralShrinking.m=0}.
\end{proof}
\end{lem}
Let $k \geq 1$ and $s\geq 0$ be integers. Given $\alpha \in P_{de,C, m}^\vee$, where $m\geq k$ let us define 
\[
N^{(k)}_s(\alpha)\coloneqq \#\left\{ \left(x^{(1)},\ldots,x^{(d-1)}\right)\in\left(P_{e-s,C,k}^{n+1}\right)^{d-1}\colon 
\begin{array}{l}
\alpha(\Psi_j(\vec{x}^{(1)},\dots, \vec{x}^{(d-1)})x)\equiv 0 \bmod t^{k+1} \\ \text{ for all }x\in P_{e+(d-1)s,C, k}\text{ and }0\leq j\leq n \end{array}
\right\}.
\]
We thus have $N^{(k)}(\alpha)=N^{(k)}_0(\alpha)$ in the notation of Sections~\ref{Se: WeylDiff.S(alpha)} and \ref{Se: WeylDiff.S(alpha,beta)}. The previous result applied $d-1$ times immediately gives the following result.
\begin{lem}\label{Le: DavShrinking}
Let $\alpha\in P_{de,C,k}^{\vee}$. For $g=0$, we have 
\[
\frac{N^{(k)}(\alpha)}{N^{(k)}_{s}\left(\alpha\right)}\le q^{\left(k+1\right)\left(d-1\right)\left(n+1\right)s}
\]
for all $s\ge0$. 

For $g\ge1$ and $s\ge\max\left(2g-1,2\right)$, we have 
\[
\frac{N^{(k)}(\alpha)}{N^{(k)}_{s}\left(\alpha\right)}\le\begin{cases}
q^{\left(k+1\right)\left(d-1\right)\left(n+1\right)s} & \text{if }g=1,\\
q^{\left(k+1\right)\left(n+1\right)\left(\left(d-1\right)s+\left(g+1\right)/2\right)} & \text{else}.
\end{cases}
\]
\end{lem}
We now aim at choosing the parameter $s$ in such a way that if $(\vec{x}^{(1)},\dots, \vec{x}^{(d-1)})$ is counted by $N^{(k)}_s(\alpha)$, then this forces $\Psi_j(\vec{x}^{(1)},\dots, \vec{x}^{(d-1)})=0$ for $0\leq j\leq n$.
\begin{lem}\label{Le: DiophApprox=0}
Let $\alpha \in P_{de,C,m}^\vee$ and $0\leq k \leq m$. Suppose $\alpha$ factors through $Z$ minimally and $(\vec{x}^{(1)},\dots, \vec{x}^{(d-1)})\in (P_{e-s, C, k}^{n+1})^{d-1}$ is given. If 
\[
s>\max\left(\frac{\deg Z-e+2g-2}{d-1},e-\frac{\deg Z}{d-1}\right)
\]
and $\alpha\left(\Psi_{j}\left(x^{(1)},\cdots,x^{(d-1)}\right)x\right)\equiv 0 \bmod t^{k+1}$ holds for all $x\in P_{e+(d-1)s, C, k}$ and $0\leq j \leq n$, then in fact
$\Psi_{j}\left(x^{(1)},\cdots,x^{(d-1)}\right)\equiv 0 \bmod t^k$ for all $0\leq j \leq n$.
\end{lem}
\begin{proof}
Write $\Psi_{j}\left(x^{(1)},\cdots,x^{(d-1)}\right)=\rho_{0}+\cdots+t^{k}\rho_{k}$
and $\alpha=\alpha_{0}+\cdots+t^{k}\alpha_{k}$, and note that then $\alpha\rho\equiv 0 \bmod t^{k+1}$
is equivalent to 
\[
\alpha_{0}\rho_{0}=\alpha_{0}\rho_{1}+\alpha_{1}\rho_{0}=\cdots=\alpha_{0}\rho_{k}+\cdots+\alpha_{k}\rho_{0}=0.
\]
By Lemma~23 of \cite{haseliu2024higher}, we have $\alpha_{0}\rho_{0}=0$
implies $\rho_{0}=0$ under the assumptions on $s$. It then follows that
$\alpha_{0}\rho_{1}=0$, and so we may apply the same reasoning to
conclude that $\rho_{1}=0$, and so forth, until we obtain $\rho_{k}=0$.
This implies $\rho=0$, as desired.
\end{proof}

Next, we bound the dimensions of the jet schemes of the variety cut out by the $\Psi_i$.
\begin{lem}\label{Le: DimEstimate}
Let $V\subset \A^{(n+1)(d-1)}$ be the variety defined by $\Psi_0=\cdots = \Psi_n=0$ and $J_k(V)$ its $k$th jet scheme. Then, 
\[
\dim J_{k}(V)\le\left(k+1\right)\left(d-1\right)\left(n+1\right)-\left(n+1\right)\left(\left\lfloor \frac{k}{d-1}\right\rfloor +1\right).
\]
\end{lem}
\begin{proof}
    We identify $x^{(i)} \in \A^{(k+1)(n+1)}$ for $1\leq i\leq d-1$ with $x_0^{(i)}+tx_1^{(i)}+\cdots + t^kx_k^{(i)}$, where $x_j^{(i)}\in\A^n$. Since $V\subset \A^{(d-1)(n+1)}$, we have $J_k(V)\subset J_k(\A^{(d-1)(n+1)})=\A^{(k+1)(n+1)(d-1)}$. Under this description, we have 
    \[
    J_k(V)=\{({x}^{(1)},\dots, {x}^{(d-1)})\in\A^{(k+1)(n+1)(d-1)}\colon \Psi_i({x}^{(1)},\dots, {x}^{(d-1)})\equiv 0 \bmod t^{k+1} \text{ for }0\leq i \leq n\}.
    \]
Let now 
\[
\Delta =\{ ({x}^{(1)},\dots, {x}^{(d-1)})\in \A^{(k+1)(n+1)(d-1)}\colon {x}^{(1)}=\cdots = {x}^{(d-1)}\}
\]
be the diagonal, which has dimension $(k+1)(n+1)$. Suppose that $({x},\dots, {x})\in\Delta \cap J_k(V)$. We must have $\nabla F({x})\equiv {0} \bmod t^{k+1}$. If ${x}\neq 0$, we can write ${x}=t^l{x}'$ with ${x}'\neq {0} \bmod t$. Then as $F$ is non-singular, this implies $l(d-1)\geq k+1$, as otherwise the reduction of ${x}'$ modulo $s$ would produce a singular point of $F$. This is equivalent to ${x}_0=\cdots = {x}_{\left\lfloor \frac{k}{d-1}\right\rfloor}=0$ and hence 
\[
\dim \Delta\cap J_k(V) \leq (n+1)\left(k-\left\lfloor \frac{k}{d-1}\right\rfloor\right).
\]

Therefore, observing that the intersection $\Delta \cap J_k (V)$ is non-empty because $\Delta$ and $J_k(V)$ are cut out by homogeneous polynomials, we obtain
\begin{align*}
    \dim J_k(V) 
    &\le \dim (\Delta \cap J_k(V)) +(k+1)(n+1)(d-1)-\dim \Delta \\
    &\le (n+1)\left(k-\left\lfloor \frac{k}{d-1}\right\rfloor\right)+(k+1)(n+1)(d-2),
\end{align*}
from which the result follows. 
\end{proof}

\begin{lem}\label{general_point_count}
Let $V$ be a variety in $\A_{K(C)}^N$, and let $D$ be a divisor on $C$. Then, the number of $K(C)$-points of $V$ such that each coordinate satisfies the norm conditions defined by $D$ is bounded above by $Cq^{\dim V h^0(C,\mathcal{O}(D))},$ where $C$ depends only on $N$ and the degree of $V$.
\end{lem}

\begin{proof}
    The proof is identical to that of Lemma 3.6 in \cite{timbook}.
    
    We may assume $V$ is geometrically irreducible. The case $\dim V =0$ is trivial. Suppose $V$ has positive dimension, then, and note that we can find some index $1\le i\le N$ such that $V$ intersects the hyperplane $H_\alpha\colon x_i = \alpha$ properly for every $\alpha$. 
    
    Let $M_V(D)$ denote the quantity we want to bound. Then, we have 
    \[M_V(D) \le \sum_{\alpha \in K(C)\colon \alpha\in H^0\left(C,\mathcal{O}(D)\right)}M_{V\cap H_\alpha}(D).\]

    By the induction hypothesis, each $M_{V\cap H_\alpha}(D) = O\left(q^{(\dim V -1)h^0(C,\mathcal{O}(D))}\right)$, so the result follows.
\end{proof}

Combining Corollaries \ref{Le: DimEstimate} and \ref{general_point_count}, we obtain the following.
\begin{cor}\label{Cor: AuxPointCount}
    Let $k\geq 0$ and $0\leq s \leq e$ be integers. We have 
    \[
    \#\left\{\vec{x}\in (P_{e-s,C, k}^{n+1})^{(d-1)}\colon 
    \begin{array}{l}
    \Psi_j(\vec{x}^{(1)},\dots, \vec{x}^{(d-1)})=0\\
    \text{for }0\leq j\leq n
    \end{array}
    \right\} \ll q^{h^0(C, L(-s\infty))(n+1)((k+1)(d-1)-\left(\lfloor \frac{k}{d-1}\rfloor +1\right))}.
    \]
\end{cor}

\subsection{Exponential sum estimates}
We now have everything at hand to present our main estimates for the sums $S(\alpha)$ and $S(\alpha,\beta)$ defined in \eqref{Eq: Def.S(alpha)} and \eqref{Eq: Def.S(alpha,beta)} respectively. 
Before we state our main results, we need to introduce some notation. Given $\alpha\in P_{de,C, m}^\vee$ such that $\alpha$ factors through $Z$ minimally, we define
\begin{equation}\label{Eq: Choose.s.J1}
     s_\alpha=\begin{cases}
       \max\left(\left\lfloor \frac{\deg Z - e + 2g -2}{d-1}\right\rfloor, \left\lfloor e - \frac{\deg Z}{d-1}\right\rfloor\right)+1 &\text{if }g=0,\\
       \max\left(\left\lfloor \frac{\deg Z - e + 2g -2}{d-1}\right\rfloor, \left\lfloor e - \frac{\deg Z}{d-1}\right\rfloor, 2g-2, 1\right)+1 &\text{if }g\geq 1.
    \end{cases}
    \end{equation}
In addition, for a non-negative integer $g$ we set 
\begin{equation}\label{Eq: Defi.f(g)}
f(g)=\begin{cases}
    0 &\text{if }g\in \{0,1\},\\
    (g+1)/2 &\text{if }g\geq 2.
\end{cases}
\end{equation}
\begin{prop}\label{Prop: S(alpha).MainEstimate}
    Let $m\geq 1$ and $\alpha \in P_{de,C, m}^\vee$. Suppose that $\alpha \sim_{\text{min}}Z$. For $s=s_\alpha$, we have 
\begin{align*}
\left(\frac{|S(\alpha)|}{\#P^{n+1}_{e,C,m}}\right)^{2^{d-2}}\ll (\#P_{e,C, m-m'}^{n+1})^{-(d-1)}&q^{(m-m'+1)(n+1)((d-1)s+f(g))}\\
\times &q^{h^0(C,L(-s\infty))(n+1)((m-m'+1)(d-1)-(\lfloor \frac{m-m'}{d-1}\rfloor +1))}    
\end{align*}
    In particular, if $e-s\geq 2g-1$, then
    \[
    \left(\frac{|S(\alpha)|}{\#P^{n+1}_{e,C,m}}\right)^{2^{d-2}}\ll q^{(m-m'+1)(n+1)f(g)-(n+1)(e-s-g+1)(\lfloor \frac{m-m'}{d-1}\rfloor +1)},
    \]
    where in both cases the implied constant is independent of $q$.
    \end{prop}

\begin{proof}
    With our choice of $s$, it is clear that Lemmas~\ref{Le: DavShrinking} and \ref{Le: DiophApprox=0} are applicable. In particular, successively applying Lemmas~\ref{Le: WeylDiff.S(alpha)} and \ref{Le: DavShrinking} with $k=m-m'$ gives
    \begin{align}
    \begin{split}\label{Eq: Ayayayay}
        \left(\frac{|S(\alpha)|}{\#P^{n+1}_{e,C,m}}\right)^{2^{d-2}}&\leq (\#P_{e,C, m-m'}^{n+1})^{-(d-1)}q^{(m-m'+1)(n+1)((d-1)s+f(g))}N^{(m-m')}_s(\alpha).
    \end{split}
    \end{align}
    By Lemma~\ref{Le: DiophApprox=0} and Corollary~\ref{Cor: AuxPointCount} we have 
    \begin{equation}\label{Eq: Bibibibib}
    N_s(\alpha)\ll q^{h^0(C,L(-s\infty))(n+1)((m-m'+1)(d-1)-(\lfloor \frac{m-m'}{d-1}\rfloor +1))},
    \end{equation}
    which once combined with \eqref{Eq: Ayayayay} verifies the first assertion of Proposition~\ref{Prop: S(alpha).MainEstimate}. For the second, note that if $e-s\geq 2g-1$, then $h^0(C,L(-s\infty))=e-s-g+1$ by Riemann--Roch. Similarly, $h^0(C, L)=e-g+1$, yielding $\#P_{e,C, m-m'}=q^{(e-g+1)(m-m'+1)}$. After canceling some of the terms, \eqref{Eq: Ayayayay} together with \eqref{Eq: Bibibibib} easily give the second assertion of Proposition~\ref{Prop: S(alpha).MainEstimate}.
\end{proof}
\begin{prop}\label{Prop: S(alpha,beta).MainEstimate}
    Let $m\geq 1$ and $\alpha,\beta \in P_{de,C, m}^\vee$. Suppose that $\alpha \sim_{\text{min}}Z$ and $\beta \sim_{\text{min}}W$. Then
    \[
    \left(\frac{|S(\alpha,\beta)|}{\#P_{e,C,m}^{2(n+1)}}\right)^{2^{d-2}}\ll q^{(n+1)R(\alpha,\beta)},
    \]
    where $R(\alpha,\beta)$ is defined to be the minimum of the quantities
    \[(m-m'+1)((d-1)(s_\alpha -(e-g+1))+f(g))+h^0(C,L(-s_\alpha \infty))\left(\left(m-m'+1\right)\left(d-1\right)-\left(\left\lfloor \frac{m-m'}{d-1}\right\rfloor +1\right)\right)\] and \[(m+1)((d-1)(s_\beta-(e-g+1))+f(g))+h^0(C,L(-s_\beta \infty))\left(\left(m+1\right)\left(d-1\right)-\left(\left\lfloor \frac{m}{d-1}\right\rfloor+1\right)\right).\]

    In particular, if $\max(e-s_\alpha,e-s_\beta)\geq 2g-1$, then 
    \[
    \left(\frac{|S(\alpha,\beta)|}{\#P_{e,C,m}^{2(n+1)}}\right)^{2^{d-2}}\ll q^{(n+1)((m+1)f(g)-M(\alpha,\beta))},
    \]
    where $M(\alpha,\beta)$ is defined to be 
\[
\begin{array}{ll}
    m'f(g)+\left(e-s_{\alpha}-g+1\right)\left\lceil \frac{m-m'+1}{d-1}\right\rceil  & \text{if }e-s_{\alpha}\ge2g-1>e-s_{\beta},\\
    \left(e-s_{\beta}-g+1\right)\left\lceil \frac{m+1}{d-1}\right\rceil  & \text{if }e-s_{\beta}\ge2g-1>e-s_{\alpha},\\
    \max\left\{ m'f(g)+\left(e-s_{\alpha}-g+1\right)\left\lceil \frac{m-m'+1}{d-1}\right\rceil ,\left(e-s_{\beta}-g+1\right)\left\lceil \frac{m+1}{d-1}\right\rceil \right\}  & \text{if }e-s_{\alpha},e-s_{\beta}\ge2g-1.
    \end{array}
\]
\end{prop}
\begin{proof}
    By Lemmas~\ref{Le: WeylDiff.S(alpha,beta)} and~\ref{Le: DavShrinking} we have 
    \begin{equation}\label{Eq: 123}
    \left(\frac{|S(\alpha,\beta)|}{\#P^{2(n+1)}_{e,C,m}}\right)^{2^{d-2}}\leq (\#P_{e,C,m-m'}^{n+1})^{-(d-1)}\min\left(N^{(m-m')}(\alpha), (\#P_{e,C, m'-1}^{n+1})^{-(d-1)}N^{(m)}(\beta)\right).
    \end{equation}
    Successively applying Lemma~\ref{Le: DavShrinking}, Lemma~\ref{Le: DiophApprox=0} and Corollary~\ref{Cor: AuxPointCount} gives 
    \begin{align}
        \begin{split}\label{Dq: 1234}
            N^{(m-m')}(\alpha)&\leq q^{(m-m'+1)(n+1)((d-1)s_\alpha +f(g))}N_{s_\alpha}^{(m-m')}(\alpha)\\
            & \ll q^{(m-m'+1)(n+1)((d-1)s_\alpha +f(g))+h^0(C, L(-s_\alpha\infty))(n+1)((m-m'+1)(d-1)-(\lfloor \frac{m-m'}{d-1}\rfloor +1))}.
        \end{split}
    \end{align}
    The same reasoning applied to $N^{(m)}(\beta)$ yields
    \begin{equation}\label{Eq: 12345}
         N^{(m)}(\beta)\ll q^{(m+1)(n+1)((d-1)s_\beta +f(g))+h^0(C,L(-s_\beta \infty))(n+1)((m+1)(d-1)-(\lfloor \frac{m}{d-1}\rfloor +1))}.
    \end{equation}
    Upon observing that $\#P_{e,C,m}=q^{(m+1)(e-g+1)}$ and $\#P_{e,C, m-m'}=q^{(m-m'+1)(e-g+1)}$, the first assertion now follows after rearranging some of the terms in the exponents. For the second, suppose that $e-s_\alpha \geq 2g-1$. Then by Riemann--Roch, we have $h^0(C,L(-s_\alpha \infty))=e-s_\alpha -g +1$, so that \eqref{Eq: 123} and \eqref{Dq: 1234} give 
    \begin{align*}
        \left(\frac{|S(\alpha,\beta)|}{\#P^{2(n+1)}_{e,C,m}}\right)^{2^{d-2}} & \ll q^{(n+1)((m-m'+1)(-(e-g+1)(d-1)+(d-1)s_\alpha +f(g) +(e-s_\alpha -g+1)(d-1))-(e-s_\alpha - g+1)(\lfloor \frac{m-m'}{d-1}\rfloor +1))}\\
        &= q^{(n+1)((m-m'+1)f(g) -(e-s_\alpha -g+1)(\lfloor \frac{m-m'}{d-1}\rfloor +1))}.
    \end{align*}
    Similarly, when $e-s_\beta\geq 2g-1$, then by Riemann--Roch we have $h^0(C,L(-s_\beta\infty))=e-s_\beta -g +1$. Thus, by \eqref{Eq: 123} and \eqref{Eq: 12345} we have 
    \begin{align*}
        \left(\frac{|S(\alpha,\beta)|}{\#P^{2(n+1)}_{e,C,m}}\right)^{2^{d-2}} & \ll  q^{(n+1)((m+1)(-(e-g+1)(d-1)+(d-1)s_\beta +f(g)+(e-s_\beta -g +1)(d-1))-(e-s_\beta -g +1)(\lfloor \frac{m}{d-1}\rfloor +1))}\\ 
        & = q^{ (n+1)((m+1)f(g)-(e-s_\beta -g+1)(\lfloor \frac{m}{d-1}\rfloor+1))}.
    \end{align*}
    The second claim now readily follows by comparing these last two estimates with the definition of $M(\alpha,\beta)$. 
\end{proof}
\section{Putting everything together}
In this final section, we apply the estimates from Propositions~\ref{Prop: S(alpha).MainEstimate} and~\ref{Prop: S(alpha,beta).MainEstimate} to prove Propositions~\ref{Prop: MinorArcGoal.J1} and~\ref{Prop: GoalMinorArc.J1(Jm)}. Having already dealt with the major arc contributions in Sections~\ref{Sec: Major.M(m)} and~\ref{Sec: Major.J1(Jm)}, this will be enough to complete the proofs of Propositions \ref{Prop: Main.Jet.Term} and \ref{Prop: Main.Jet.Can}, and hence of   main Theorems~\ref{Thm: Main.CanSing} and~\ref{Thm: Main.TermSing} 
\subsection{Canonical singularities}\label{Sec: PuttingEverythingTogether.Can}
Given $n,d,g,e\in \N$ satisfying the hypotheses of Theorem~\ref{Thm: Main.CanSing}, our goal is to show that 
\begin{equation}\label{Eq: CanSing.MinorArcs}
\#P_{e,C,m}^{-(n+1)}\sum_{\deg Z \geq e-2g+2}\sum_{\alpha \sim_{\text{min}}Z}|S(\alpha)| \to 0\quad \text{as }q\to\infty.
\end{equation}
By Lemma~\ref{Le: DirichletApprox} any $\alpha \in P_{de,C,m}^\vee$ factors through some effective divisor $Z\subset C$ of degree at most $de/2+1$. As there are $O_{d,e,g}(1)$ possibilities for integers $D$ satisfying $e-2g+2\leq D \leq de/2+1$, it suffices to consider the contribution from $\deg Z =D$ for a fixed $D$ in \eqref{Eq: CanSing.MinorArcs}. We shall denote this contribution by $N(D)$.

Let 
\begin{equation}\label{Eq: Defi.s.CanSing}
    s= \begin{cases}
        \max\left(\left\lfloor \frac{D-e+2g-2}{d-1}\right\rfloor, \left\lfloor e-\frac{D}{d-1}\right\rfloor \right)+1 &\text{if }g=0,\\
        \max\left(\left\lfloor \frac{D-e+2g-2}{d-1}\right\rfloor, \left\lfloor e-\frac{D}{d-1}\right\rfloor, 2g-2,1 \right)+1 &\text{if }g\geq 1
    \end{cases}
\end{equation}
and observe that the first term in the maximum dominates the second if $de/2-g+1<D$ and the second dominates the first if $D\leq de/2-g+1$. Suppose that $e-s\geq 2g -1$. Under this assumption, Proposition~\ref{Prop: S(alpha).MainEstimate} hands us the estimate 
\begin{equation}\label{Eq: UpperBound.S(alpha).minor.cansing}
\frac{|S(\alpha)|}{\#P^{n+1}_{e,C,m}}\ll q^{(n+1)((m-m'+1)f(g)-(e-s-g+1)(\lfloor \frac{m-m'}{d-1}\rfloor+1))/2^{d-2}}
\end{equation}
whenever $\alpha$ factors through $Z$ minimally with $\deg Z = D$. The number of $\alpha \in P_{de,C, m}$ that factor through a given $Z$ is $\#P^\vee_{de,C, m-1}$ times the number of $\alpha_0\in P_{de,C}^\vee$ that factor through $Z$, so that in total there are $O(q^{D+m(de+1-g)})$ available $\alpha$. In addition, $C$ is a smooth curve over $\F_q$ and hence has $O(q^D)$ effective divisors of degree $D$. In particular, combining these observations with \eqref{Eq: UpperBound.S(alpha).minor.cansing} shows that 
\[
N(D) \ll q^{2D+m(de+1-g)+(n+1)((m-m'+1)f(g)-(e-s-g+1)(\lfloor \frac{m-m'}{d-1}\rfloor +1 ))/2^{d-2}}.
\]
In particular, if we define 
\begin{equation}\label{Eq: Defi.A}
A\coloneqq 2^{d-2}\frac{2D+m(de+1-g)}{(e-s-g+1)(\lfloor\frac{m-m'}{d-1}\rfloor +1) -(m-m'+1)f(g)},
\end{equation}
and $e-s\geq 2g-1$ holds, then to establish Proposition~\ref{Prop: MinorArcGoal.J1} it suffices to show that $n+1>A$.
\subsubsection{$d\geq 3$, $g=0$ and $1\leq e \leq d-2$}
Recall that in this case we assume $n+1>2^{d-2}d(2d+1)$ when $e=1$ and $n+1>2^{d-1}(d-1)(de+2)$ when $e\geq 2$. Note also $e-s=\lceil D/(d-1)\rceil -1\geq 2g -1$, so that 
\begin{align*}
A&=2^{d-2}\frac{2D+m(de+1)}{\left\lceil\frac{D}{d-1}\right\rceil(\left\lfloor \frac{m-m'}{d-1}\right\rfloor +1)}\\
&\leq 2^{d-2}\frac{2D+m(de+1)}{\max(1, D/(d-1))\max(1, m/(2(d-1)))}, 
\end{align*}
since $\lfloor (m-m')/(d-1)\rfloor +1 =\lceil (m-m'+1)/(d-1)\rceil \geq \lfloor (m+1)/2\rfloor /(d-1)\geq m/(2(d-1))$ and $\lceil D/(d-1)\rceil \geq \max(1,D/(d-1)$. For any $D$, the last fraction is increasing in $m$ when $1\leq m \leq 2(d-1)$ and decreasing when $m>2(d-1)$. Therefore, it is maximal at $m=2(d-1)$ and hence 
\begin{equation}\label{Eq: AAAAAAAAAAAAAAAAA}
A\leq 2^{d-1}\frac{D+(d-1)(de+1)}{\max(1,D/(d-1))}.
\end{equation}
When $e=1$, then we assume that $d\geq 4$, implying $D\leq d/2+1\leq d-1$. In this case we thus have 
\[
A\leq 2^{d-2}(d+2+2(d-1)(d+1))=2^{d-2}d(2d+1),
\]
which is sufficient. When $e\geq 2$, then \eqref{Eq: AAAAAAAAAAAAAAAAA} is maximal at $D=d-1$, so that 
\[
A\leq 2^{d-1}(d-1)(de+2),
\]
which is again satisfactory.
\subsubsection{$d\geq 3$}
Recall that if $d\geq 3$ and $g\geq 1$, we assume that $e>e_0$, where
\begin{equation}\label{Eq: Assumption.e.minor.arcs}
    e_0\coloneqq  (g-1)(2^{d-1}(d-1)^2(2d-1)+2)+(d-1)(g+(d-1)f(g))(2^{d-1}(d-1)(d^2-d+1)+1),
\end{equation}
while we assume $e\geq d-2$ when $g=0$. We begin by showing that if $d\geq 3$, then in fact 
\[
s=\max\left(\left\lfloor\frac{D-e+2g-2}{d-1}\right\rfloor, \left\lfloor e-\frac{D}{d-1}\right\rfloor\right)+1
\]
also holds for $g\geq 1$. First, suppose $D\leq de/2-g+1$. Then
\[
\left\lfloor e-\frac{D}{d-1}\right\rfloor \geq e-1-\frac{de/2-g+1}{d-1} = \frac{e(d/2-1)-d+g}{d-1}.
\]
Our assumption on $e$ implies that $e\geq 2^{d-1}(d-1)^2(1+f(g))$ and a straightforward computation reveals that 
\[
\frac{2^{d-1}(d-1)^2(1+f(g))(d/2-1)-d+g}{d-1}\geq \max\{2g-2,1\}. 
\]
Second, if $de/2-g+1<D$, then 
\[
\left\lfloor \frac{D-e+2g-2}{d-1}\right\rfloor \geq \frac{e(d/2-1)+g-1}{d-1}\geq \frac{e(d/2-1)+g-d}{d-1}\geq \max\{2g-2,1\}
\]
by what we have shown in the first case and hence completes the verification of the claim. 

Next, we show that 
\begin{equation}\label{Eq: CanSing.Minor.e-s.large}
    e-s\geq 2g-1.
\end{equation}
If $D\leq de/2-g+1$, then we have 
\begin{equation}\label{DDDDDDD}
e-s = \left\lceil \frac{D}{d-1}\right\rceil -1 \geq \frac{e-2g+2}{d-1}-1.
\end{equation}
If $g=0$, then $e\geq 1$ and the right hand side is at least $3/(d-1)-1\geq -1$, which is sufficient. If $g\geq 1$, then \eqref{Eq: Assumption.e.minor.arcs} implies that $(e-2g+2)/(d-1)\geq 2^{d-1}(d-1)(1+f(g))\geq 2g-1$, which is also satisfactory. When $D>de/2-g+1$, then 
\[
e-s=e-\left\lfloor \frac{D-e+2g-2}{d-1}\right\rfloor -1 \geq e-\frac{e(d/2-1)+2g-1}{d-1}-1=\frac{ed/2+2g-d}{d-1}.
\]
When $g=0$, this is at least $-1$ because $d\geq 2$. When $g\geq 1$, then $2g\geq 3-2g$ holds, which together with $d\geq 3>2$ implies 
\[
e-s\geq \frac{e-2g+2}{d-1}-1,
\]
which is precisely the inequality \eqref{DDDDDDD}. We have thus already shown that this implies $e-s\geq 2g-1$.

In particular, to establish \eqref{Eq: CanSing.MinorArcs} it is enough to show $n+1>A$, where $A$ is defined in \eqref{Eq: Defi.A}.

{\textit{Case I:} $D\leq de/2 -g +1$.} Note that we are always in this case when $g=0$. Here, $e-s+1=\lceil D/(d-1)\rceil$ and hence 
\begin{align}
\begin{split}\label{Fire}
A &\leq 2^{d-2}(d-1)\frac{2D+m(de+1-g)}{(D -g(d-1))(\lfloor (m-m')/(d-1)\rfloor +1)-(m-m'+1)(d-1)f(g)}.
\end{split}
\end{align}
Let us first treat the case $m\leq 2(d-1)$. This implies $(m-m')/(d-1)= (\lfloor (m+1)/2 \rfloor -1)/(d-1)<1$ and hence $\lfloor (m-m')/(d-1)\rfloor =0$. Using $m\leq 2(d-1)$, the inequality \eqref{Fire} gives 
\begin{align}
A &\leq 2^{d-1}(d-1)\frac{D+(d-1)(de+1-g)}{D-g(d-1)-(m-m'+1)(d-1)f(g)}\nonumber\\
& \leq 2^{d-1}(d-1)\frac{D+(d-1)(de+1-g)}{D-g(d-1)-(d-1)^2f(g)}.\label{CaseIAAAAAAAA}
\end{align}
Note that the last expression is decreasing in $D$ and hence is maximal at $D=e-2g+2$, which implies that
\begin{align}
    A &\leq 2^{d-1}(d-1)\frac{e-2g+2+(d-1)(de+1-g)}{e-2g+2-g(d-1)-(d-1)^2f(g)}.\label{Rain}
\end{align}
When $g=0$, then the right hand side is strictly increasing in $e$, whence 
\begin{equation}
A < 2^{d-1}(d-1)(d^2-d+1).
\end{equation}
Thus, $A<n+1$ when $g=0$. If $g\geq 1$, then the right hand side of \eqref{Rain} is strictly decreasing in $e$ and one can check that when evaluated at $e_0$ defined in \eqref{Eq: Assumption.e.minor.arcs}
it is identical to $2^{d-1}(d-1)(d^2-d+1)+1$. Of course, here we have chosen $e_0$ to ensure this identity. As we assume that $e>e_0$, this shows that $A< n+1$, which is satisfactory.

Next, let us treat the opposite case $m>2(d-1)$. As $\lfloor (m-m')/(d-1)\rfloor +1 \geq (m-m'+1)/(d-1)$, from \eqref{Fire} we get 
\begin{align*}
A& \leq 2^{d-2}(d-1)^2\frac{2D +m(de+1-g)}{(m-m'+1)(D-g(d-1)-f(g)(d-1)^2)}\\
 &\leq 2^{d-1}(d-1)^2\frac{2D+m(de+1-g)}{m(D-g(d-1)-f(g)(d-1)^2)},
\end{align*}
where we used that $m-m'+1=\lfloor (m+1)/2\rfloor \geq m/2$. The last expression is strictly decreasing in $m$. In particular, $m>2(d-1)$ implies that 
\[
A< 2^{d-1}(d-1)\frac{D+(d-1)(de+1-g)}{D-g(d-1)-f(g)(d-1)^2}.
\]
This is precisely the term in \eqref{CaseIAAAAAAAA} and we have already seen that this implies $n+1>A$ when we treated the case $m\leq 2(d-1)$. 

\textit{Case II:} $de/2-g+1 < D \leq de/2+1$. In this case we have $e-s+1=e-\lfloor (D-e+2g-2)/(d-1)\rfloor$, so that 
\begin{align}
    A&= 2^{d-2}\frac{2D+m(de+1-g)}{(e-\lfloor \frac{D-e+2g-2}{d-1}\rfloor -g)(\lfloor \frac{m-m'}{d-1}\rfloor +1) -(m-m'+1)f(g)}\nonumber\\
    & \leq 2^{d-2}\frac{de+2+m(de+1-g)}{(e-\frac{de/2-e+2g-1}{d-1}-g)(\lfloor \frac{m-m'}{d-1}\rfloor+1)-(m-m'+1)f(g)}\nonumber\\
    &= 2^{d-2}(d-1)\frac{de+2+m(de+1-g)}{(de/2-2g+1-g(d-1))(\lfloor \frac{m-m'}{d-1}\rfloor +1)-(m-m'+1)f(g)(d-1)}.\label{Eq: FFFFFFFFFFFFFFFF}
\end{align}
Let us first assume that $m\leq 2(d-1)$. Then 
\[
A\leq 2^{d-2}(d-1)\frac{de+2+2(d-1)(de+1-g)}{de/2-2g+1-g(d-1)-f(g)(d-1)^2}.
\]
Let us denote the right hand side by $h(e)$. Then $h(e)$ is a strictly decreasing function of $e$ for each $g\geq 1$. As we assume that $e>e_0$, where $e_0$ is defined in \eqref{Eq: Assumption.e.minor.arcs}, we thus have $A<h(e_0)$. With a rather involved computation or a computer algebra system one can check that $h(e_0)\leq 2^{d-1}(d-1)(d^2-d+1)$. This verifies that $A<n+1$. 

Next, we turn to the case $m>2(d-1)$. Then \eqref{Eq: FFFFFFFFFFFFFFFF} yields 
\begin{align}
A &\leq 2^{d-2}(d-1)^2\frac{de+2+m(de+1-g)}{(m-m'+1)(de/2-2g+1-g(d-1)-f(g)(d-1)^2)}\nonumber\\
&\leq 2^{d-1}(d-1)^2\frac{de+2+m(de+1-g)}{m(de/2-2g+1-g(d-1)-f(g)(d-1)^2)}, \label{FAFAFAFA}   
\end{align}
where we used that $m-m'+1=\lfloor (m+1)/2\rfloor \geq m/2$. The term in \eqref{FAFAFAFA} is strictly decreasing in $m$, and as we assume that $m>2(d-1)$, we get 
\[
A< 2^{d-2}(d-1)\frac{de+2+2(d-1)(de+1-g)}{de/2-2g+1-g(d-1)-f(g)(d-1)^2}=h(e).
\]
As we have already shown that $h(e)<n+1$ under our assumptions on $e$, this completes our treatment of Case II. 
\subsubsection{$d=2$}
Recall that we only treat $d=2$ when $g\geq 1$, in which case we assume $e>e_0$, where
\begin{equation}\label{Eq: Defi.e0.d=2}
    e_0=19g+7f(g)-3.
\end{equation}
Note that when $d=2$, the definition of $s$ simplifies to 
\[
s= \max\left( D-e+2g-2, e-D, 2g-2, 1\right)+1
\]
and $D$ satisfies $e+2-2g\leq D \leq e+1$. If $e+2-2g \leq D \leq e$, then $D-e+2g-2\leq 2g-2$ and $e-d\leq 2g-2$, so that
\[
s=\begin{cases}
    2&\text{if }g=1,\\
    2g-1 &\text{if }g\geq 2.
\end{cases}
\]
While if $D=e+1$, then $s=\max(2g-1,-1, 2g-2,1)+1=2g$. In either case we have $e-s \geq e-2g \geq e_0-2g\geq 2g-1$ by \eqref{Eq: Defi.e0.d=2}. In particular, it suffices to show that $n+1>A$ with $A$ defined in \eqref{Eq: Defi.A}. It transpires from \eqref{Eq: Defi.A} that $A$ is maximal when $D=e+1$ and $s=2g$. Therefore, we have 
\[
A\leq \frac{2(e+1)+m(2e+1-g)}{(m+1-m')(e-3g+1-f(g))}.
\]
When $m=2$, then $m+1-m'=\lfloor (m+1)/2\rfloor =1$ and hence 
\begin{equation}\label{Eq: GenericCase.d=2.minor}
A \leq 2\frac{3e+2-g }{e-3g+1-f(g)}.    
\end{equation}
Let us denote the right hand side by $h(e)$. Then $h(e)$ is strictly increasing in $e$ for any $g\geq 1$. As we assume $e>e_0$, this implies 
\[
A < 2\frac{56g+21f(g)-7}{16g+6f(g)-2}=7.
\]
As we assume $n+1\geq 2^{d-1}(d-1)(d^2-d+1)+1=7$, this completes our treatment of the case $d=2$ and hence of Theorem~\ref{Thm: Main.CanSing}.

\subsection{Terminal singularities}\label{Sec: PuttingEverythingTogether.Term}

Given $n,d,g,e\in \N$ satisfying the hypotheses of Theorem \ref{Thm: Main.TermSing}, our goal is to show that 
\begin{equation}\label{Goal: TermSing}
    \#P_{e,C,m}^{-2(n+1)}\sum_{\substack{\alpha,\beta \in P_{de,C, m}^\vee \\ \max(\deg \alpha, \deg \beta) \geq e-2g +2}}|S(\alpha,\beta)|\to 0 \quad\text{ as }q\to \infty. 
\end{equation}
Let  us write $\deg(\alpha)=D_{\alpha}$ and $\deg(\beta)=D_{\beta}$. By Lemma~\ref{Le: DirichletApprox} we have $D_\alpha,D_\beta \leq de/2+1$. Recall the definitions of $s_\alpha$ and $s_\beta$ in Proposition \ref{Prop: S(alpha,beta).MainEstimate}] and assume that $\max(e-s_\alpha,e-s_\beta)\geq 2g-1$. Then, by Proposition~\ref{Prop: S(alpha,beta).MainEstimate}
 we have 
 \[
 \frac{|S(\alpha,\beta)|}{\#P_{e,C, m}^{2(n+1)}}\ll q^{(n+1)((m+1)f(g)-M(\alpha,\beta))/2^{d-2}}. 
 \]
There are at most $q^{D_{\alpha}+m\left(de-g+1\right)}$ linear
functionals on $P_{de,C,m}$ that factor through any specific degree
$D_{\alpha}$ effective divisor and $O\left(q^{D_{\alpha}}\right)$ such effective divisors
(likewise for $D_{\beta}$). As there are $O_{d,e}(1)$ possibilities for $D_\alpha$ and $D_\beta$, to verify \eqref{Goal: TermSing} it suffices to show \[
2\left(D_{\alpha}+D_{\beta}+m\left(de-g+1\right)\right)+\left((m+1)f(g)-M(\alpha,\beta)\right)\left(n+1\right)/2^{d-2}<0,
\]
which, if we set 
\[
A'=2^{d-1}\frac{D_{\alpha}+D_{\beta}+m\left(de-g+1\right)}{M(\alpha,\beta)-(m+1)f(g)},
\]
is equivalent to 
\[
n+1>A'.
\]
Under the assumptions of Theorem~\ref{Thm: Main.TermSing}, we will show that we always have $\max(e-s_\alpha,e-s_\beta)\geq 2g-1$ and verify the inequality $n+1>A'$. This will therefore complete our proof of Theorem \ref{Thm: Main.TermSing}.
\subsubsection{$d\geq 4$, $g=0$ and $1\leq e \leq d-1$}
Note that for $g=0$, we have 
\begin{align*}
M(\alpha,\beta)&=\max\left(\left\lceil \frac{D_\alpha}{d-1}\right\rceil \left\lceil\frac{m-m'+1}{d-1}\right\rceil, \left\lceil\frac{D_\beta}{d-1}\right\rceil \left\lceil \frac{m+1}{d-1}\right\rceil\right)\\
    &\geq \max\left(\left\lceil\frac{D_\alpha}{d-1}\right\rceil \max\left(1,\frac{m}{2(d-1)}\right), \left\lceil \frac{D_\beta}{d-1}\right\rceil \max\left(1, \frac{m+1}{d-1}\right)\right).
\end{align*}
In particular, if $\max(D_\alpha,D_\beta)\leq d-1$, then we obtain 
\[
A'\leq 2^{d-1}\frac{D_\alpha+D_\beta +m(de+1)}{\max(1, (m+1)/(d-1))}.
\]
Let us first deal with the case $e=1$ and $d\ge 4$. For $D_\alpha+D_\beta\ge d+1=de+1$, the right-hand side of the above inequality is maximized at $m=d-2$, where it is equal to \[2^{d-1}(D_\alpha+D_\beta + (d-2)(d+1)).\]
Note that $\max(D_\alpha,D_\beta)\le d/2+1,$ so $D_\alpha+D_\beta \ge d+1$ precisely when $D_\alpha = D_\beta = (d+1)/2$ for $d$ odd and when $(D_\alpha,D_\beta)\in {(d/2+1, d/2), (d/2,d/2+1), (d/2+1,d/2+1)}$ for $d$ even.  

In either case, for $D_\alpha + D_\beta \ge d+1$, it follows that \[A'\le 2^{d-1}(d+2+(d-2)(d+1)).\]
For $D_\alpha + D_\beta < d+1$, the right-hand side is instead an increasing function of $m$, so \[A'\le 2^{d-1}(d-1)(d+1).\]

So, for $e=1$ and $d\ge 4$, we have \[A'\le 2^{d-1}d^2 \le 2^{d-2}d(2d+1)<n+1\] by our assumption on $n$ in this case.

Note that 
\[
\min(de/2+1,d-1)=\begin{cases} d/2+1 &\text{if $d\geq 4$ and $e=1$}\\
d-1 &\text{else.}\end{cases}
\]

We shall therefore assume $e>1$ from now on. Then, if $\max(D_\alpha, D_\beta)\le d-1$, we again have
\[A'\le 2^{d-1}\frac{D_\alpha + D_\beta + m(de+1)}{\max(1,(m+1)/(d-1))}\le 2^{d-1}(d-1)(de+1)\eqqcolon B_0,\]
since $D_\alpha + D_\beta \le 2d-2 < de+1$.

Now, suppose $\max(D_\alpha,D_\beta) \ge d$. This gives
\[
A'\leq 2^{d-1}(d-1)\frac{D_\alpha+D_\beta+m(de+1)}{\max\left(D_\alpha \max(1,m/(2(d-1))), D_\beta \max(1, (m+1)/(d-1))\right)}.
\]
When $1\leq m \leq d-2$, then we get 
\begin{equation}\label{Eq: e<=d-3.1}
A'\leq 2^{d-1}(d-1)\frac{2\max(D_\alpha,D_\beta)+(d-2)(de+1)}{\max(D_\alpha,D_\beta)}\leq 2^{d-1}(d-1)(2+(d-2)(de+1)/d)\eqqcolon B_1.
\end{equation}
While if $d-1\leq m \leq 2d-2$, then we obtain
\[
A'\leq 2^{d-1} \frac{D_\alpha+D_\beta+m(de+1) }{\max\left( D_\alpha, D_\beta \frac{m+1}{d-1}\right)}.
\]
The right-hand side is increasing in $D_\beta$ when $D_\beta \leq (d-1)D_\alpha/(m+1)$ and decreasing otherwise. In particular, it is maximal at $D_\beta=(d-1)D_\alpha/(m+1)$. Note that this implies $D_\alpha\geq D_\beta$ and hence $D_\alpha\geq d$. Thus, 
\[
    A'\leq 2^{d-1}(d-1)\frac{D_\alpha(1 +(d-1)/(m+1))+m(de+1)}{D_\alpha }\leq 2^{d-1}(d-1)((1+(d-1)/(m+1))+m(de+1)/d).
\]
The last quantity is increasing in $m$ and hence maximal at $m=2(d-1)$. We obtain 
\begin{equation}\label{Eq: e<=d-3.2}
    A'\leq 2^{d-1}(d-1)(1+(d-1)/(2d-1)+2(d-1)(de+1)/d)\eqcolon B_2,
\end{equation}
say. Next, we consider the case $2d-1\leq m$. Here we have 
\[
A'\leq 2^{d-1}(d-1)^2\frac{D_\alpha+D_\beta +m(de+1) }{\max\left( D_\alpha m/2, D_\beta(m+1)\right)}.
\]
The quantity is maximal when $D_\beta = D_\alpha \frac{m}{2(m+1)}$ and $D_\alpha\geq D_\beta$, so that again $D_\alpha \geq d$ must hold. We get 
\begin{align*}
    A'&\leq 2^{d}(d-1)^2\frac{D_\alpha (1+m/(2(m+1)))+m(de+1)}{D_\alpha m}\\
    &\leq 2^{d}(d-1)^2\frac{d (1+m/(2(m+1)))+m(de+1)}{d m}\\
    &\leq 2^{d}(d-1)^2\frac{d (1+(2d-1)/(4d))+(2d-1)(de+1)}{d (2d-1)},
\end{align*}
where we used that the expression in the second line is decreasing in $m$. A straightforward computation reveals that we always have 
\[
2^{d}(d-1)^2\frac{d (1+(2d-1)/(4d)+(2d-1)(de+1)}{d (2d-1)} < B_2
\]
and also $B_0,B_1 < B_2$. It thus transpires from \eqref{Eq: e<=d-3.1} and \eqref{Eq: e<=d-3.2} that we have 
\[
A' \leq B_2 = 2^{d-1}(d-1)\left(1+\frac{d-1}{2d-1}+\frac{2(d-1)(de+1)}{d}\right) < n+1,
\]
by our assumption on $n$. This concludes our treatment of this case. 
\subsubsection{$d\ge3$}

\paragraph{Case I: $e-2g+2\le D_{\beta}\le de/2-g+1$}

\subparagraph{I.1: $D_{\beta}\ge D_{\alpha}/2$}

Note that $e-s_\beta = \left\lceil \frac{D_{\beta}}{d-1}\right\rceil - 1$, which by the assumptions on $e$, is clearly at least $2g-1$.
We have 
\[
M\left(\alpha,\beta\right)\ge\left(e-s_{\beta}-g+1\right)\left\lceil \frac{m+1}{d-1}\right\rceil =\left(\left\lceil \frac{D_{\beta}}{d-1}\right\rceil -g\right)\left\lceil \frac{m+1}{d-1}\right\rceil ,
\]
so
\[
A'\le2^{d-1}\left(d-1\right)\frac{2D_{\beta}+D_{\beta}+m\left(de-g+1\right)}{\left(D_{\beta}-(d-1)g\right)\left\lceil \frac{m+1}{d-1}\right\rceil -(d-1)(m+1)f(g)}.
\]
If $m+1\le d-1$, the right-hand side is 
\begin{align*}
 & 2^{d-1}\left(d-1\right)\frac{3D_{\beta}+m\left(de-g+1\right)}{D_{\beta}-(d-1)g-(d-1)(m+1)f(g)}\\
 & \le2^{d-1}\left(d-1\right)\frac{3D_{\beta}+\left(d-2\right)\left(de-g+1\right)}{D_{\beta}-(d-1)g-(d-1)^{2}f(g)}\\
 & \le2^{d-1}\left(d-1\right)\frac{3\left(e-2g+2\right)+\left(d-2\right)\left(de-g+1\right)}{\left(e-2g+2\right)-(d-1)g-(d-1)^{2}f(g)}\\
 & \eqqcolon S.
\end{align*}

Note that 
\begin{align*}
 & \frac{3\left(e-2g+2\right)+\left(d-2\right)\left(de-g+1\right)}{\left(e-2g+2\right)-(d-1)g-(d-1)^{2}f(g)}\\
 & =\left(3+d(d-2)\right)-\frac{-\left(3+d(d-2)\right)\left(2g-2+(d-1)g+(d-1)^{2}f(g)\right)+3\left(2g-2\right)+\left(d-2\right)\left(g-1\right)}{\left(e-2g+2\right)-(d-1)g-(d-1)^{2}f(g)}\\
 & =\left(3+d(d-2)\right)-\frac{-((d-2)d+3)(d-1)^{2}f(g)+d(d(-dg+g+2)-5)+g+2}{\left(e-2g+2\right)-(d-1)g-(d-1)^{2}f(g)}.
\end{align*}
Let $T=-((d-2)d+3)(d-1)^{2}f(g)+d(d(-dg+g+2)-5)+g+2$, so that the
above is 
\[
\left(3+d(d-2)\right)-\frac{T}{\left(e-2g+2\right)-(d-1)g-(d-1)^{2}f(g)}.
\]

For $g=0$, $T=d(2d-5)+2>0$. In this case, $A'<2^{d-1}\left(d-1\right)\left(3+d(d-2)\right)=2^{d-1}\left(d-1\right)\left(d^{2}-2d+3\right)$
for all $e$, so $n+1>A'$. 

For $g\ge1$, we have $T\le-d^{3}+3d^{2}-5d+3<0$. In this case, 
\[
2^{d-1}\left(d-1\right)\left(\left(3+d(d-2)\right)+\frac{-T}{\left(e_{0}-2g+2\right)-(d-1)g-(d-1)^{2}f(g)}\right)< 2^{d-1}\left(d-1\right)\left(d^{2}-2d+3\right)+1,
\]
 by computer verification.

If $m+1>d-1$, then 

\begin{align*}
A' & \le2^{d-1}\left(d-1\right)\frac{3D_{\beta}+m\left(de-g+1\right)}{\left(D_{\beta}-(d-1)g\right)\left\lceil \frac{m+1}{d-1}\right\rceil -(d-1)(m+1)f(g)}\\
 & \le2^{d-1}\left(d-1\right)^{2}\frac{3D_{\beta}+m\left(de-g+1\right)}{\left(m+1\right)\left(D_{\beta}-(d-1)g-(d-1)^{2}f(g)\right)}\\
 & \eqqcolon S'.
\end{align*}

We have 
\[
de-g+1-3D_{\beta}\text{ is }\begin{cases}
<0 & \text{if }D_{\beta}>\frac{de-g+1}{3},\\
=0 & \text{if }D_{\beta}=\frac{de-g+1}{3},\\
>0 & \text{if }D_{\beta}<\frac{de-g+1}{3}.
\end{cases}
\]

If $D_{\beta}>\left(de-g+1\right)/3$, then we have 
\[
S'<2^{d-1}\left(d-1\right)\frac{3D_{\beta}+\left(d-2\right)\left(de-g+1\right)}{D_{\beta}-(d-1)g-(d-1)^{2}f(g)},
\]
which our bound for $S$ earlier deals with already. 

If $D_{\beta}=(de-g+1)/3$, then we have 
\begin{align*}
S' & =3\cdot2^{d-1}\left(d-1\right)^{2}\frac{de-g+1}{\left(de-g+1\right)-3(d-1)g-3(d-1)^{2}f(g)}\\
 & =3\cdot2^{d-1}\left(d-1\right)^{2}\left(1+\frac{3(d-1)g+3(d-1)^{2}f(g)}{\left(de-g+1\right)-3(d-1)g-3(d-1)^{2}f(g)}\right).
\end{align*}

Let $T=3(d-1)g+3(d-1)^{2}f(g).$

For $g=0$, we have $S'=3\cdot2^{d-1}\left(d-1\right)^{2}< 2^{d-2}\left(d-1\right)\left(4d^{2}-4d+3\right)$, which is satisfactory.

For $g\ge1$, we have 
\begin{align*}
 & 3\cdot2^{d-1}\left(d-1\right)^{2}\left(1+\frac{3(d-1)g+3(d-1)^{2}f(g)}{\left(de_{0}-g+1\right)-3(d-1)g-3(d-1)^{2}f(g)}\right)\\
 & < 2^{d-2}\left(d-1\right)\left(4d^{2}-4d+3\right)+1
\end{align*}
by plugging in $e_{0}$ using a computer algebra system.

If $D_{\beta}<(de-g+1)/3$, then we have 
\begin{align*}
S' & <2^{d-1}\left(d-1\right)^{2}\frac{de-g+1}{D_{\beta}-(d-1)g-(d-1)^{2}f(g)}\\
 & \le2^{d-1}\left(d-1\right)^{2}\frac{de-g+1}{e-2g+2-(d-1)g-(d-1)^{2}f(g)}\\
 & =2^{d-1}\left(d-1\right)^{2}\left(d-\frac{g-1-2dg+2d-d(d-1)g-d(d-1)^{2}f(g)}{e-2g+2-(d-1)g-(d-1)^{2}f(g)}\right)\\
 & =2^{d-1}\left(d-1\right)^{2}\left(d-\frac{T}{e-2g+2-(d-1)g-(d-1)^{2}f(g)}\right),
\end{align*}

where $T=g-1-2dg+2d-d(d-1)g-d(d-1)^{2}f(g)$. 

For $g=0$, we have $T=-1+2d>0$, so $S'<2^{d-2}\left(d-1\right)\left(4d^{2}-4d+3\right)$. 

For $g\ge1$, we have $T<0$. Then, 
\[
S'< 2^{d-2}\left(d-1\right)\left(4d^{2}-4d+3\right)+1
\]
for $e\ge e_{0}$ by using a computer algebra system.

\subparagraph{I.2: $D_{\beta}<D_{\alpha}/2$}

Let us use 
\[
M(\alpha,\beta)\ge m'f(g)+(e-s_{\alpha}-g+1)\left\lceil \frac{m-m'+1}{d-1}\right\rceil .
\]

\subparagraph{I.2.1: $D_{\alpha}\le de/2-g+1$}

We have 
\[
e-s_{\alpha}-g+1=\left\lceil \frac{D_{\alpha}}{d-1}\right\rceil -g,
\]
which by the assumptions on $e$, is clearly at least $2g-1-g+1$.
So 
\begin{align*}
A' & =2^{d-1}\frac{D_{\alpha}+D_{\beta}+m\left(de-g+1\right)}{M(\alpha,\beta)-(m+1)f(g)}\\
 & \le2^{d-1}\frac{D_{\alpha}+D_{\alpha}/2+m\left(de-g+1\right)}{\left(\left\lceil \frac{D_{\alpha}}{d-1}\right\rceil -g\right)\left\lceil \frac{m-m'+1}{d-1}\right\rceil -\left(m-m'+1\right)f(g)}\\
 & \le2^{d-1}\left(d-1\right)\frac{3D_{\alpha}/2+m\left(de-g+1\right)}{\left(D_{\alpha}-(d-1)g\right)\left\lceil \frac{m-m'+1}{d-1}\right\rceil -(d-1)\left(m-m'+1\right)f(g)}.
\end{align*}
Note that this is dominated by case \eqref{eq:asfkawefbuaweff} because
in that case we take $D_{\alpha}$ to be even smaller than allowed
in this case.

\subparagraph{I.2.2: $de/2-g+1<D_{\alpha}\le de/2+1$}

Here, note that we are implicitly assuming $g\ge1$ and $D_{\alpha}>2D_{\beta}$
as well.

We have 
\[
e-s_{\alpha}-g+1=e-\left\lfloor \frac{D_{\alpha}-e+2g-2}{d-1}\right\rfloor -g,
\] which by the assumptions on $e$, is clearly at least $2g-1-g+1$.
So 
\begin{align*}
A' & \le2^{d-2}\frac{3D_{\alpha}+2m\left(de-g+1\right)}{\left(e-\left\lfloor \frac{D_{\alpha}-e+2g-2}{d-1}\right\rfloor -g\right)\left\lceil \frac{m-m'+1}{d-1}\right\rceil -(m-m'+1)f(g)}\\
 & \le2^{d-2}(d-1)\frac{3\left(\frac{de}{2}+1\right)+2m\left(de-g+1\right)}{\left(\frac{de}{2}-g+1-gd\right)\left\lceil \frac{m-m'+1}{d-1}\right\rceil -(d-1)(m-m'+1)f(g)}.
\end{align*}
If $m\le2(d-1)$, then $m-m'+1\le d-1$, so 
\begin{align*}
A' & \le2^{d-2}\left(d-1\right)\frac{3(de/2+1)+4(d-1)\left(de-g+1\right)}{\frac{de}{2}-g+1-gd-(d-1)^{2}f(g)}\\
 & =2^{d-2}\left(d-1\right)\left(3+8(d-1)-\frac{T}{\frac{de}{2}-g+1-gd-(d-1)^{2}f(g)}\right)\\
 & \eqqcolon C,
\end{align*}
where $T=-4+5f(g)-8d^{3}f(g)+d^{2}\left(21f(g)-8g\right)+g+d\left(4-18f(g)+g\right)$. 

For $g\ge1$, we have $T<0$. Plugging in $e\ge e_{0}$, we notice
that 
\[
A'<2^{d-2}\left(d-1\right)\left(4d^{2}-4d+3\right)+1
\]
by using a computer algebra system.

If $m>2(d-1)$, then 
\begin{align*}
A' & \le2^{d-2}(d-1)^{2}\frac{3\left(\frac{de}{2}+1\right)+2m\left(de-g+1\right)}{(m-m'+1)\left(\left(\frac{de}{2}-g+1-gd\right)-(d-1)^{2}f(g)\right)}\\
 & \le2^{d-1}(d-1)^{2}\frac{3\left(\frac{de}{2}+1\right)+2m\left(de-g+1\right)}{m\left(\left(\frac{de}{2}-g+1-gd\right)-(d-1)^{2}f(g)\right)}\\
 & \le2^{d-2}(d-1)\frac{3\left(\frac{de}{2}+1\right)+4(d-1)\left(de-g+1\right)}{\left(\frac{de}{2}-g+1-gd\right)-(d-1)^{2}f(g)}\\
 & =C,
\end{align*}
so we reduce to the $m\le2(d-1)$ case.

\paragraph{Case II: $de/2-g+1<D_{\beta}\le de/2+1$}

Implicitly, we are assuming $g\ge1$. 

Here, let us take the $\beta$ estimate: 
\[
e-s_{\beta}-g+1=e-\left\lfloor \frac{D_{\beta}-e+2g-2}{d-1}\right\rfloor -g,
\] which by the assumptions on $e$, is clearly at least $2g-1-g+1$.
So
\begin{align*}
A' & =2^{d-1}\frac{D_{\alpha}+D_{\beta}+m\left(de-g+1\right)}{M(\alpha,\beta)-(m+1)f(g)}\\
 & \le2^{d-1}\frac{D_{\alpha}+D_{\beta}+m\left(de-g+1\right)}{\left(e-\left\lfloor \frac{D_{\beta}-e+2g-2}{d-1}\right\rfloor -g\right)\left\lceil \frac{m+1}{d-1}\right\rceil -\left(m+1\right)f(g)}\\
 & \le2^{d-1}(d-1)\frac{de+2+m\left(de-g+1\right)}{\left(\frac{de}{2}-g+1-gd\right)\left\lceil \frac{m+1}{d-1}\right\rceil -\left(d-1\right)\left(m+1\right)f(g)}.
\end{align*}
If $m+1\le d-1$, then 
\begin{align*}
A' & \le2^{d-1}\left(d-1\right)\frac{de+2+\left(d-2\right)\left(de-g+1\right)}{\frac{de}{2}-g+1-gd-\left(d-1\right)^{2}f(g)}\\
 & \le2^{d-1}\left(d-1\right)\frac{3(de/2+1)+4(d-1)\left(de-g+1\right)}{\frac{de}{2}-g+1-gd-\left(d-1\right)^{2}f(g)}\\
 & =C,
\end{align*}
so we can reduce to I.2.2.

If $m+1>d-1$, then 
\begin{align*}
A' & \le2^{d-1}(d-1)\frac{de+2+m\left(de-g+1\right)}{\left(\frac{de}{2}-g+1-gd\right)\left\lceil \frac{m+1}{d-1}\right\rceil -\left(d-1\right)\left(m+1\right)f(g)}\\
 & \le2^{d-1}\left(d-1\right)^{2}\frac{\left(de-g+1\right)+\frac{g-1}{m+1}}{\frac{de}{2}-g+1-gd-\left(d-1\right)^{2}f(g)}
\end{align*}
and the right hand side is non-increasing in $m$. Hence, 
\[
A'\le2^{d-1}\left(d-1\right)\frac{de+2+\left(d-2\right)\left(de-g+1\right)}{\frac{de}{2}-g+1-gd-\left(d-1\right)^{2}f(g)}\le C,
\]
so we can again reduce to I.2.2.

\paragraph{Case III: $D_{\beta}\le e-2g+1$}

Since we are working with minor arcs, this implies $D_{\alpha}\ge e-2g+2$. 

\subparagraph{III.1: $de/2-g+1<D_{\alpha}\le de/2+1$}

Note that we are implicitly assuming $g\ge1$. Let us use the $\alpha$
estimate: We have 
\begin{align*}
M(\alpha,\beta) & \ge m'f(g)+\left(e-s_{\alpha}-g+1\right)\left\lceil \frac{m-m'+1}{d-1}\right\rceil \\
 & =m'f(g)+\left(e-\left\lfloor \frac{D_{\alpha}-e+2g-2}{d-1}\right\rfloor -g\right)\left\lceil \frac{m-m'+1}{d-1}\right\rceil.
\end{align*}
Note that $e-s_\alpha \ge 2g-1$ by the assumptions on $e$. We have
\begin{align*}
A' & =2^{d-1}\frac{D_{\alpha}+D_{\beta}+m\left(de-g+1\right)}{M(\alpha,\beta)-(m+1)f(g)}\\
 & \le2^{d-1}\frac{D_{\alpha}+e-2g+1+m\left(de-g+1\right)}{\left(e-\left\lfloor \frac{D_{\alpha}-e+2g-2}{d-1}\right\rfloor -g\right)\left\lceil \frac{m-m'+1}{d-1}\right\rceil -\left(m-m'+1\right)f(g)}\\
 & \le2^{d-1}\left(d-1\right)^{2}\frac{D_{\alpha}+e-2g+1+m\left(de-g+1\right)}{\left(m-m'+1\right)\left(de-D_{\alpha}+2-(d+1)g-(d-1)^{2}f(g)\right)}\\
 & \le2^{d-1}\left(d-1\right)^{2}\frac{\frac{de}{2}+1+e-2g+1+m\left(de-g+1\right)}{\left(m-m'+1\right)\left(\frac{de}{2}+1-(d+1)g-(d-1)^{2}f(g)\right)}.
\end{align*}
If $m\le2(d-1)$, then $m-m'+1\le d-1$, so 
\begin{align*}
A' & \le2^{d-1}\left(d-1\right)\frac{\frac{de}{2}+1+e-2g+1+2\left(d-1\right)\left(de-g+1\right)}{\frac{de}{2}+1-(d+1)g-(d-1)^{2}f(g)}\\
 & =2^{d-1}\left(d-1\right)\left(4d-3+\frac{2}{d}-\frac{T}{\frac{de}{2}+1-(d+1)g-(d-1)^{2}f(g)}\right)\\
 & \eqqcolon C',
\end{align*}
where $T=-3+7f(g)-4d^{3}f(g)+d^{2}\left(11f(g)-4g\right)+g+d\left(2-12f(g)-g\right)-2\left(-1+f(g)+g\right)/d$. 

For $g\ge1$, we have $T<0$. Plugging in $e\ge e_{0}$, we notice
that 
\[
A'<2^{d-2}\left(d-1\right)\left(4d^{2}-4d+3\right)+1
\]
by using a computer algebra system.

If $m>2(d-1)$, then 
\begin{align*}
A' & \le2^{d-1}\left(d-1\right)^{2}\frac{\frac{de}{2}+1+e-2g+1+m\left(de-g+1\right)}{\left(m-m'+1\right)\left(\frac{de}{2}+1-(d+1)g-(d-1)^{2}f(g)\right)}\\
 & \le2^{d}\left(d-1\right)^{2}\frac{\frac{de}{2}+1+e-2g+1+m\left(de-g+1\right)}{m\left(\frac{de}{2}+1-(d+1)g-(d-1)^{2}f(g)\right)}\\
 & \le2^{d}\left(d-1\right)^{2}\frac{\frac{de}{2}+1+e-2g+1+2(d-1)\left(de-g+1\right)}{2(d-1)\left(\frac{de}{2}+1-(d+1)g-(d-1)^{2}f(g)\right)}\\
 & =2^{d-1}\left(d-1\right)\frac{\frac{de}{2}+1+e-2g+1+2(d-1)\left(de-g+1\right)}{\frac{de}{2}+1-(d+1)g-(d-1)^{2}f(g)}\\
 & =C',
\end{align*}
so we can reduce to the $m\le2(d-1)$ case.

\subparagraph{III.2: $e-2g+2\le D_{\alpha}\le de/2-g+1$}

\subparagraph{III.2.1: $D_{\alpha}/2>D_{\beta}$}

First, consider the case where $D_{\alpha}/2>D_{\beta}$, where we
use the $\alpha$ estimate. Then, we have 
\[
M(\alpha,\beta)\ge m'f(g)+\left(e-s_{\alpha}-g+1\right)\left\lceil \frac{m-m'+1}{d-1}\right\rceil =m'f(g)+\left(\left\lceil \frac{D_{\alpha}}{d-1}\right\rceil -g\right)\left\lceil \frac{m-m'+1}{d-1}\right\rceil.
\]
Note that $e-s_\alpha\ge 2g-1$ by the assumptions on $e$.
So 
\begin{align*}
A' & =2^{d-1}\frac{D_{\alpha}+D_{\beta}+m\left(de-g+1\right)}{M(\alpha,\beta)-(m+1)f(g)}\\
 & \le2^{d-1}\frac{D_{\alpha}+D_{\alpha}/2+m\left(de-g+1\right)}{\left(\left\lceil \frac{D_{\alpha}}{d-1}\right\rceil -g\right)\left\lceil \frac{m-m'+1}{d-1}\right\rceil -\left(m-m'+1\right)f(g)}\\
 & \le2^{d-1}\left(d-1\right)\frac{3D_{\alpha}/2+m\left(de-g+1\right)}{\left(D_{\alpha}-(d-1)g\right)\left\lceil \frac{m-m'+1}{d-1}\right\rceil -(d-1)\left(m-m'+1\right)f(g)}.
\end{align*}

If $m\le2\left(d-1\right)$, we obtain 
\begin{equation}
A'\le2^{d-1}\left(d-1\right)\frac{3D_{\alpha}/2+2(d-1)\left(de-g+1\right)}{\left(D_{\alpha}-(d-1)g\right)-(d-1)^{2}f(g)}.\label{eq:asfkawefbuaweff}
\end{equation}
The right hand side is decreasing in $D_{\alpha}$, so it is at most 
\begin{align*}
E & \coloneqq2^{d-1}\left(d-1\right)\frac{3\left(e-2g+2\right)/2+2(d-1)\left(de-g+1\right)}{e-2g+2-(d-1)g-(d-1)^{2}f(g)}\\
 & =2^{d-1}\left(d-1\right)\left(\frac{3}{2}+2d(d-1)+\frac{(d-1)\left(4d^{2}g+\left(4d^{3}-8d^{2}+7d-3\right)f(g)+4d(g-2)-g+4\right)/2}{e-2g+2-(d-1)g-(d-1)^{2}f(g)}\right).
\end{align*}
Let $T=-(d-1)\left(4d^{2}g+\left(4d^{3}-8d^{2}+7d-3\right)f(g)+4d(g-2)-g+4\right)/2,$
so that 
\[
E=2^{d-1}\left(d-1\right)\left(\frac{3}{2}+2d(d-1)-\frac{T}{e-2g+2-(d-1)g-(d-1)^{2}f(g)}\right).
\]
For $g=0$, we have $T=(d-1)(4d-2)>0$ for $d\ge3$. So $A'\le E<2^{d-2}\left(d-1\right)\left(4d^{2}-4d+3\right)$. 

For $g\ge1$, we have $T<0$ for $d\ge3$, and note that 
\[
2^{d-1}\left(d-1\right)\left(\frac{3}{2}+2d(d-1)-\frac{T}{e_{0}-2g+2-(d-1)g-(d-1)^{2}f(g)}\right)<2^{d-2}\left(d-1\right)\left(4d^{2}-4d+3\right)+1
\]
for $e_{0}=2^{d-2}\left(d-1\right)^{2}\left(4d^{2}g+\left(4d^{3}-8d^{2}+7d-3\right)f(g)+4d(g-2)-g+4\right)-1+(d+1)g+(d-1)^{2}f(g).$ 

If $m>2(d-1)$, we obtain 
\begin{align*}
A' & \le2^{d-1}\left(d-1\right)\frac{3D_{\alpha}/2+m\left(de-g+1\right)}{\left(D_{\alpha}-(d-1)g\right)\left\lceil \frac{m-m'+1}{d-1}\right\rceil -(d-1)\left(m-m'+1\right)f(g)}\\
 & \le2^{d-1}\left(d-1\right)^{2}\frac{3D_{\alpha}/2+m\left(de-g+1\right)}{\left(m-m'+1\right)\left(D_{\alpha}-(d-1)g-(d-1)^{2}f(g)\right)}\\
 & \le2^{d-1}\left(d-1\right)^{2}\frac{3D_{\alpha}/2+m\left(de-g+1\right)}{\left(m/2\right)\left(D_{\alpha}-(d-1)g-(d-1)^{2}f(g)\right)}\\
 & =2^{d-1}\left(d-1\right)\frac{3D_{\alpha}\cdot\frac{d-1}{m}+2(d-1)\left(de-g+1\right)}{D_{\alpha}-(d-1)g-(d-1)^{2}f(g)}\\
 & \le2^{d-1}\left(d-1\right)\frac{3D_{\alpha}/2+2(d-1)\left(de-g+1\right)}{\left(D_{\alpha}-(d-1)g\right)-(d-1)^{2}f(g)},
\end{align*}
but this is the same as \eqref{eq:asfkawefbuaweff}.

\subparagraph{III.2.2: $D_{\alpha}/2\le D_{\beta}$}

Next, consider the case where $D_{\alpha}/2\le D_{\beta}$, where
we use the $\beta$ estimate instead. Then, we have
\[
M(\alpha,\beta)\ge\left(e-s_{\beta}-g+1\right)\left\lceil \frac{m+1}{d-1}\right\rceil =\left(\left\lceil \frac{D_{\beta}}{d-1}\right\rceil -g\right)\left\lceil \frac{m+1}{d-1}\right\rceil.
\]
By the assumptions on $e$, we have $e-s_\beta \ge 2g-1$.
so 
\begin{align*}
A' & =2^{d-1}\frac{D_{\alpha}+D_{\beta}+m\left(de-g+1\right)}{M(\alpha,\beta)-(m+1)f(g)}3D_{\beta}\frac{d-1}{m+1}+\frac{m}{m+1}\left(de-g+1\right)\\
 & \le2^{d-1}\frac{3D_{\beta}+m\left(de-g+1\right)}{\left(\left\lceil \frac{D_{\beta}}{d-1}\right\rceil -g\right)\left\lceil \frac{m+1}{d-1}\right\rceil -\left(m+1\right)f(g)}\\
 & \le2^{d-1}\left(d-1\right)\frac{3D_{\beta}+m\left(de-g+1\right)}{\left(D_{\beta}-(d-1)g\right)\left\lceil \frac{m+1}{d-1}\right\rceil -(d-1)\left(m+1\right)f(g)}.
\end{align*}
If $m\le d-2$, we obtain 
\begin{equation}
A'\le2^{d-1}\left(d-1\right)\frac{3D_{\beta}+(d-2)\left(de-g+1\right)}{\left(D_{\beta}-(d-1)g\right)-(d-1)^{2}f(g)}.\label{eq:asdfawe}
\end{equation}
Note that $D_{\beta}\ge D_{\alpha}/2\ge e/2-g+1$. The RHS is decreasing
in $D_{\beta}$, so it is at most 
\begin{align*}
E & \coloneqq2^{d-1}\left(d-1\right)\frac{3\left(e/2-g+1\right)+(d-2)\left(de-g+1\right)}{e/2-g+1-(d-1)g-(d-1)^{2}f(g)}\\
 & =2^{d-1}\left(d-1\right)\frac{3\left(e-2g+2\right)+2(d-2)\left(de-g+1\right)}{e-2g+2-2(d-1)g-2(d-1)^{2}f(g)}\\
 & =2^{d-1}\left(d-1\right)\left(3+2d(d-2)+\frac{2\left(\left(2d^{2}-4d+3\right)\left(d-1\right)^{2}f(g)+2d^{3}g-4d^{2}g-2d^{2}+2dg+5d-g-2\right)}{e-2g+2-2(d-1)g-2(d-1)^{2}f(g)}\right).
\end{align*}
Let $T=-2\left(\left(2d^{2}-4d+3\right)\left(d-1\right)^{2}f(g)+2d^{3}g-4d^{2}g-2d^{2}+2dg+5d-g-2\right),$
so that 
\[
E=2^{d-1}\left(d-1\right)\left(3+2d(d-2)-\frac{T}{e-2g+2-2(d-1)g-2(d-1)^{2}f(g)}\right).
\]
For $g=0$, we have $T=-2\left(-2d^{2}+5d-2\right)>0$ for $d\ge3$.
So $A'\le E<2^{d-1}\left(d-1\right)\left(2d^{2}-4d+3\right)$. 

For $g\ge1$, we have $T<0$ for $d\ge3$. Furthermore, for $e\ge e_{0}$,
we have $E<2^{d-2}\left(d-1\right)\left(4d^{2}-4d+3\right)+1$ by
computer verification.

If $m>d-2$, we obtain 
\begin{align*}
A' & \le2^{d-1}\left(d-1\right)\frac{3D_{\beta}+m\left(de-g+1\right)}{\left(D_{\beta}-(d-1)g\right)\left\lceil \frac{m+1}{d-1}\right\rceil -(d-1)\left(m+1\right)f(g)}\\
 & \le2^{d-1}\left(d-1\right)^{2}\frac{3D_{\beta}+m\left(de-g+1\right)}{\left(m+1\right)\left(D_{\beta}-(d-1)g-(d-1)^{2}f(g)\right)}\\
 & \le2^{d-1}\left(d-1\right)^{2}\frac{\frac{3D_{\beta}}{m+1}+\frac{m}{m+1}\left(de-g+1\right)}{D_{\beta}-(d-1)g-(d-1)^{2}f(g)}.
\end{align*}

For $g=0$ and $d=3$, we have $3D_{\beta}\ge de-g+1$. Thus the right hand side
is decreasing as a function of $m$ and
\begin{align*}
A' & \le2^{d-1}\left(d-1\right)^{2}\frac{\frac{3D_{\beta}}{d-1}+\frac{d-2}{d-1}\left(de-g+1\right)}{D_{\beta}-(d-1)g-(d-1)^{2}f(g)}\\
 & =2^{d-1}\left(d-1\right)\frac{3D_{\beta}+\left(d-2\right)\left(de-g+1\right)}{D_{\alpha}-(d-1)g-(d-1)^{2}f(g)},
\end{align*}
which is precisely \eqref{eq:asdfawe}.

For $g\ge1$ or $d\ge4$, since $D_{\beta}\le e-2g+1$, we have $3D_{\beta}<de-g+1$,
which means that 
\begin{align*}
A' & <2^{d-1}\left(d-1\right)\frac{de-g+1}{D_{\beta}-(d-1)g-(d-1)^{2}f(g)}\\
 & \le2^{d-1}\left(d-1\right)\frac{de-g+1}{\frac{e}{2}-g+1-(d-1)g-(d-1)^{2}f(g)}\\
 & =2^{d-1}\left(d-1\right)\left(2d-\frac{T}{\frac{e}{2}-g+1-(d-1)g-(d-1)^{2}f(g)}\right),
\end{align*}
where $T=-1-2d\left(-1+f(g)\right)-2d^{3}f(g)+d^{2}\left(4f(g)-2g\right)+g.$ 

For $g=0$ and $d\ge4$, we have $T>0$, so $A'<2^{d-1}\left(d-1\right)\left(2d\right)<2^{d-2}\left(d-1\right)\left(4d^{2}-4d+3\right).$ 

For $g=1$ and $d\ge3$, we have $T<0$. For $e\ge e_{0}$, we have
\[
2^{d-1}\left(d-1\right)\left(2d-\frac{T}{\frac{e}{2}-g+1-(d-1)g-(d-1)^{2}f(g)}\right)\le2^{d-2}\left(d-1\right)\left(4d^{2}-4d+3\right)+1
\]
by using a computer algebra system.
\subsubsection{$d=2$ and $g\geq 1$}
Note that in this case $(\alpha,\beta)$ belongs to the minor arcs precisely when $e-2g+2\leq \max(D_\alpha,D_\beta)\leq e+1$. Let $\gamma\in \{\alpha,\beta\}$ and suppose that $e-2g+2\leq D_\gamma\leq e+1$. Then $s_\gamma$ simplifies to 
\begin{equation}\label{Eq: s_gamma.d=2}
s_\gamma =
\begin{cases} 
    2g &\text{if $D_\gamma =e+1$},\\
    2 &\text{if $D_\gamma =e$ and $g=1$},\\
    2g-1 &\text{if $e-2g+2\leq D_\gamma \leq e$ and $g\geq 2$}. 
\end{cases}
\end{equation}
Our assumption that $e>29g-5+14f(g)$ readily implies $e-s_\gamma \geq 2g-1$ in these cases, so that it suffices to show $A'<n+1$. 
\subparagraph{I: $e+2-2g\leq D_\beta \leq e+1$} 
Note that by \eqref{Eq: s_gamma.d=2} we have $M(\alpha,\beta)\geq (e-3g+1)(m+1)$, so that 
\begin{align*}
    A'&\leq 2\frac{D_\alpha+D_\beta+m(2e-g+1)}{    (m+1)(e-3g+1-f(g))}\\
    &\leq 2\frac{2(e+1)+m(2e-g+1)}{(m+1)(e-3g+1-f(g))}.
\end{align*}
The last term is decreasing in $m$ and hence maximal at $m=1$, yielding 
\[
A'\leq \frac{4e+3-g}{e-3g+1-f(g)}.
\]
This is decreasing in $e$ and with a computer algebra system one can check that when evaluated at $e=e_0$ is strictly smaller than 12, which is sufficient. 
\subparagraph{II: $2g \leq D_\beta \leq e+1-2g$}
Note that when $D_\beta<e+2-2g$, then $e-s_\beta+1=D_\beta$ and we still have $e-s_\beta \geq 2g-1$, since $D_\beta \geq 2g$. As $D_\beta \leq e+1-2g$ implies $D_\alpha \geq e+2-2g$, we have 
\[
M(\alpha,\beta)\geq \max(D_\beta(m+1), m'f(g)+(e-3g+1)(m-m'+1))\eqcolon B.
\]
We obtain 
\[
A'\leq 2\frac{D_\alpha+D_\beta +m(2e-g+1)}{B-(m+1)f(g)}.
\]
One can check that the last term is decreasing in $D_\beta$ if $B=D_\beta(m+1)$ and increasing if $B=m'f(g)+(e-2g+1)(m-m'+1)$. Thus it is maximal when the two terms in the definition of $B$ agree, in which case $D_\alpha \leq e+1$ yields
\begin{align*}
    A' &\leq 2\frac{e+1+(m'f(g)+(e-3g+1)(m-m'+1))/(m+1)+m(2e-g+1)   }{(m-m'+1)(e-3g+1-f(g))}\\
    & \leq \frac{2}{e-3g+1-f(g)}\left( \frac{e+1}{m-m'+1}+\frac{f(g)}{m}+\frac{e-3g+1}{m+1}+2(2e-g+1)\right),
\end{align*}
where we used that $m'/((m+1)(m-m'+1))\leq m$. The last term is decreasing in $m$ and hence maximal at $m=1$. Thus 
\begin{align*}
    A'&\leq \frac{2(e+1+f(g))+e-3g+1+2(4e-2g+2)}{e-3g+1-f(g)}\\
    &= \frac{11e+7+12f(g)-7g}{e-3g+1-f(g)}.
\end{align*}
One can check that when evaluated at $e=29g-5+14f(g)$ this is strictly less than 12, which is sufficient. 
\subparagraph{III: $D_\beta <2g$} 
In this case we have $M(\alpha,\beta)\geq m'f(g)+(e-3g+1)(m-m'+1)$. Using $m-m'+1=\lfloor (m+1)/2\rfloor$ and $D_\beta\leq 2g-1$, we get 
\begin{align*}
    A'&\leq 2\frac{2g-1+D_\alpha +m(2e-g+1) }{\lfloor \frac{m+1}{2}\rfloor (e-3g+1-f(g)} \\
    &\leq 2\frac{2g+e+m(2e-g+1)}{\lfloor \frac{m+1}{2}\rfloor (e-3g+1-f(g))}.
\end{align*}
The last expression is decreasing as $m$ runs through even positive integers and increasing for odd values of $m$. Plugging in $m=2$ and evaluating the limit as $m\to \infty$ for odd $m$, we obtain 
\begin{align*}
    A' &\leq \frac{2}{e-3g+1-f(g)}\max(4e-2g+2,5e+2)\\
    &= 2\frac{5e+2}{e-3g+1-f(g)}.
\end{align*}
Evaluating this expression at $e=29g-5+14f(g)$ again shows that $A'<12$, which is sufficient. 

\printbibliography

\end{document}